\documentclass[a4paper,10pt]{amsart}
\usepackage[utf8]{inputenc}
\usepackage{latexsym, amsfonts, amsthm, amsmath, amssymb, tikz, enumitem, booktabs, float,color,multicol,epsf, url,epsfig,epstopdf, array, setspace, graphicx}
\usetikzlibrary{matrix, calc, arrows}

\theoremstyle{plain}
\newtheorem{thm}{Theorem}[section]

\newtheorem{prop}{Proposition}[section]

\newtheorem{lem}{Lemma}[section]

\title{Contact forms with large systolic ratio in arbitrary dimensions}
\author{Murat Sa\u glam}

\begin{document}

\begin{abstract}
If a contact form on a $(2n+1)$-dimensional closed contact manifold  admits closed Reeb orbits, then its systolic ration is defined to be the quotient of $(n+1)$th power of the shortest period of Reeb orbits by the contact volume. We prove that every co-orientable contact structure on any closed contact manifold admits a contact form with arbitrarily large systolic ratio. This statement generalizes the recent result of Abbondandolo et al. in dimension three to higher dimensions. We extend the plug construction of Abbondandolo et. al. to any dimension, by means of generalizing the hamiltonian disc maps  studied by the authors to the symplectic ball of any dimension. The plug is a mapping torus and it is equipped with a special contact form so that one can use it to modify a given contact form if the Reeb flow leads to a circle bundle on a "large" portion of the given contact manifold. Inserting the plug sucks up the contact volume while the minimal period remains the same. Following the ideas of Abbondandolo et al. and using Giroux's theory of Liouville open books, we show that any co-orientable contact structure is defined by a contact form, which is suitable to be modified via inserting plugs.   
\end{abstract}

\maketitle

\section{Introduction}
One of the classical problems in Riemannian geometry is to give an upper bound on the length of the shortest non-constant closed geodesic in terms of the Riemannian area on a given closed surface. More specifically, one studies the scaling invariant functional
\begin{equation}\label{sysratio}
 \rho(S,g)=\frac{l_{{\rm min}}(S,g)^2}{{\rm area}(S,g)},   
\end{equation}
on the space of all Riemannian metrics on a given closed surface $S$. Here, $l_{{\rm min}}(S,g)$ denotes the length of the shortest non-constant $g$-
geodesic and ${\rm area}(S,g)$ denotes the area of $S$ with respect to the metric $g$. 

In 1949, Loewner showed that if in (\ref{sysratio}), $l_{{\rm min}}(S, g)$ is replaced by ${\rm sys}_1(S,g)$, namely the length of a shortest non-contractible geodesic, the corresponding ratio $\rho_{{\rm nc}}(\mathbb{T}^2,\cdot)$ admits an optimal bound. In 1952, Pu proved the existence of an optimal bound on $\rho_{{\rm nc}}(\mathbb{RP}^2,\cdot)$. In fact, in both statements the metrics that maximize $\rho_{{\rm nc}}$ do not admit any contractible geodesic and hence they also maximize $\rho$. In early 80's, Gromov proved that 
$$\rho_{{\rm nc}}(S,\cdot)\leq 2$$
for any non-simply connected closed surface $S$ but this bound is in general non-optimal  \cite{Gromov}. On the other hand, in late 80's Croke gave the first upper bound on $\rho(S^2\cdot)$ \cite{croke}, which was later improved by several authors.  

Actually in \cite{Gromov}, Gromov studied the so called \textit{ systolic ratio} in any dimension and showed that for any essential n-dimensional closed manifold $M$,
$$\rho_{{\rm nc}}(M,g)=\frac{{\rm sys}_1(M,g)^n}{{\rm vol}(M,g)}$$
admits an upper bound, which depends only on the dimension. Another natural direction for the generalization of the problem is weakening the Riemannian assumption on the metric. In fact,
the ratios $\rho$ and $\rho_{{\rm nc}}$ generalize to the Finsler setting by replacing the Riemannian area with the Holmes-Thompson area and the bounds on $\rho$  generalize to the Finsler case \cite{APBT}. For the detailed account of results about the systolic ratio in Riemannian and Finsler geometry, we refer to \cite{3sphere} and \cite{general}.  

The systolic ratio $\rho$ naturally generalizes to contact geometry. A \emph{contact manifold} is a $(2n+1)$-dimensional manifold $V$ equipped with a maximally non-integrable hyperplane distribution $\xi$. That is, $\xi$ is locally given by the kernel of a $1$-form $\alpha$ so that $\alpha\wedge(d\alpha)^n$ is non-vanishing. In this case, one says $\xi$ is a \emph{contact structure} on $V$. If $\xi$ is co-orientable, then there exists a global $1$-form $\alpha$, referred as a \emph{contact form on $(V,\xi)$}, such that $\ker \alpha=\xi$. We note that if $\alpha'$ is another contact form on $(V,\xi)$ then $\alpha=f\alpha'$ for some non-vanishing function $f$ on $V$. \emph{In this paper, we will be interested in only the contact structures that are co-orientable.} A contact form $\alpha$ gives rise to a natural dynamical system. Namely, one defines the \emph{Reeb vector field} $R_\alpha$ on $V$ via 
$$\imath_{R_\alpha}d\alpha=0\;\,{\rm and}\;\,\imath_{R_\alpha}\alpha=1.$$ 
Then the \emph{contact systolic ratio} on a closed contact manifold $(V,\xi)$ is defined to be the scaling invariant functional 
$$\rho(V,\alpha):=\frac{T_{{\rm min}}(V,\alpha)^{n+1}}{{\rm vol}(V,\alpha)}$$
on the space of all contact forms on $(V,\xi)$. Here, $T_{{\rm min}}(V,\alpha)$ denotes the minimum among the periods of all orbits of the Reeb vector field $R_\alpha$ and 
$${\rm vol}(V,\alpha):=\int_V \alpha \wedge (d\alpha)^n$$
is the contact volume of $V$ associated to the contact form $\alpha$. 

We note that the contact systolic ratio is not merely a generalization of the notion to a dynamical system but it is strongly related to the previous setting. In fact, given a smooth Finsler manifold $(M,F)$, the canonical Liouville 1-form $pdq$ on the cotangent bundle $T^*S$, restricts to a contact form $\alpha_F$  on the unit cotangent bundle $S^*_FM$. In this case, the Reeb flow is nothing but the geodesic flow restricted to $S^*_F M$ and up to a universal constant, the contact volume  ${\rm vol}(S^*_F M,\alpha_F)$ is the Holmes-Thompson volume of $(M,F)$. Hence the contact systolic ratio of $(S^*_F M,\alpha_F)$ recovers the classical systolic ratio of $(M,F)$.

But it turns out that it is not possible to bound the contact systolic ratio globally. In the case of the tight 3-sphere $(S^3,\xi_{{\rm st}})$, it was shown in \cite{3sphere} that the systolic ratio can be made arbitrarily large. Yet it was also shown that the Zoll contact forms, namely the contact forms for which all Reeb orbits are closed and share the same minimal period, are  maximizers of the functional $\rho(S^3,\cdot)$ if the functional is restricted to a $C^3$-neighbourhood of all Zoll contact forms. For any contact 3-manifold $(M,\xi)$, the non-existence of a global bound on $\rho(M,\cdot)$ is later proved by the same authors in \cite{general} and in \cite{Benedetti}, the local bound on $\rho(S^3,\cdot)$ was generalized to all contact 3-manifolds that admit Zoll contact forms.  

The aim of this paper is to prove that the contact systolic ratio is unbounded in any dimension. Here we need to point that $\rho(V,\alpha)$ makes sense only if the Reeb vector field $R_\alpha$ admits a closed orbit. If $\dim V=3$, by a result of Taubes \cite{Taubes}, we know that any contact form on $V$ admits a closed Reeb orbit but in higher dimensions, this might not be the case. Since we aim for the non-existence of a bound on $\rho$, it is legitimate for us to ignore this issue. 

We first prove that on any odd dimensional sphere, which is equipped with \emph{the standard contact structure}, see Section 4, the systolic ratio is unbounded, see Theorem \ref{bigspheres}.
The strategy of our proof is precisely the same as the proof of the corresponding 3-dimensinal result in \cite{3sphere}. The standard contact structure on the sphere $S^{2n+1}$ admits a contact form $\alpha_0$, so called the \emph{standard contact form}, for which the Reeb flow leads to the Hopf fibration. In this case, the complement of a codimension two submanifold of $S^{2n+1}$ may be identified with $S^1\times \mathbb{B}$, where  $\mathbb{B}$ is the open $2n$-dimensional unit ball, in such a way that the Reeb flow leads to the trivial $S^1$-fibration over $\mathbb{B}$, see the proof of Theorem \ref{bigspheres}. We note that  $S^1\times \mathbb{B}$ may be viewed as the mapping torus of the identity map on the ball. With this motivation, we  recover $S^1\times \mathbb{B}$ as the mapping torus associated to a compactly supported symplectomorphism on  $\mathbb{B}$ so that it admits a contact form $\alpha$ that fits to $\alpha_0$ near the boundary and the associated Reeb dynamics is described by the dynamics of the symplectomorphism. Namely, the closed Reeb orbits translate into the periodic points and periods of closed Reeb orbits are described by the so called \emph{action} of the symplectomorphism. Moreover, the contact volume is, up to a certain shift, given by the \emph{Calabi invariant} of the symplectomorphism, see Section 2. Having this "dictionary" at hand, we construct a suitable isotopy of compactly supported symplectomorphisms of the ball, where the action and the Calabi invariant are arranged for our purposes. As a result, we get \textit{the plug}, that is  a contact manifold $(S^1\times \mathbb{B},\beta)$ such that $\beta$ coincides with $\alpha_0$ near the boundary of $S^1\times \mathbb{B}$ and the periods of orbits of $R_\beta$ are bounded away from zero but the contact volume is arbitrarily small, see Lemma 3.1. As an immediate corollary of the existence of the plug, we show that $\rho(S^{2n+1},\cdot)$ is unbounded on the set of contact forms that define the standard contact structure. In order to construct the isotopy required for the plug, we extend the construction of radial hamiltonians on the unit disc \cite{3sphere} to the unit ball of arbitrary dimension, which requires nothing but some minor adjustments. 

In the last section, we prove our main result. Namely, we show that the systolic ratio is unbounded on a closed contact manifold of any dimension, see Theorem \ref{bigmanifolds}. Again we use the strategy in the proof of the corresponding 3-dimensional result in \cite{general}. On a given contact manifold, we construct a contact form, for which the Reeb flow leads to an $S^1$-fibration on a "large portion" of the manifold and away from this portion the periods of closed Reeb orbits are bounded away from zero. Then we fill this large portion with suitably resized plugs so that the most of the contact volume is eaten up but the minimal period is still bounded away from zero. The construction of the contact form, which is  of Boothby-Wang type on a large portion of the manifold, is the main result of this paper, see Proposition \ref{niceform}. Our construction mimics the corresponding 3-dimensional statement in \cite{general} and relies on the results of Giroux on higher dimensional open books \cite{openbook,ILD}. But it also requires an inductive argument, which uses the corresponding 3-dimensional statement, the Proposition 1 in \cite{general}, as its basis step and contains many technical aspects which do not appear in dimension three.      
\newline  
\textbf{Acknowledgements.} This work is part of a project in the SFB/TRR 191 `Symplectic Structures in Geometry,
Algebra and Dynamics', funded by the DFG.

\section{Contact mapping tori of the symplectomorphisms of the unit ball}
In this section, we present the construction of contact mapping tori associated to isotopies of compactly supported symplectomorphisms of the unit ball, where the Reeb dynamics and the contact volume are described in terms of the dynamics and the classical invariants of the underlying symplectomorphisms. We first recall these classical invariants in a context that is sufficient for our purposes. 

Let $m\geq 2$ be given and  $\mathbb{B}\subset \mathbb{R}^{2m}$ be the $2m$-dimensional \textit{open} unit ball centered at the origin. We fix the standart symplectic form 
$$\omega=\sum_{i=1}^m dx_i\wedge dy_i$$
via the coordinates $z_i=x_i+iy_i$, $(z_1,...,z_m)\in \mathbb{C}^m\cong \mathbb{R}^{2m}$ and we fix the primitive $$\lambda=\frac{1}{2}\sum_{i=1}^m x_idy_i-y_idx_i.$$ 
Let $\varphi\in {\rm Diff}_c(\mathbb{B},\omega)$ be a compactly supported symplectomorphism. Then $\varphi^*\lambda-\lambda$ is a closed one form and therefore it is exact. This allows us to find a unique function, called \textit{the action of $\varphi$}, 
$$\sigma_{\varphi,\lambda}:\mathbb{B}\rightarrow \mathbb{R}$$
that satisfies
\begin{eqnarray}\label{action_general}
d\sigma_{\varphi,\lambda}=\varphi^*\lambda-\lambda \textrm{ and } \sigma_{\varphi,\lambda}=0 \textrm{ near }\partial \mathbb{B}.
\end{eqnarray}
As suggested by the notation, the action depends on the primitive of $\omega$, for the details in the case of $m=1$ see \cite{3sphere}. The average of the action, namely
\begin{eqnarray}\label{calabi_general}
{\rm CAL}(\varphi):=\int_W \sigma_{\varphi,\lambda} \,\omega^m
\end{eqnarray}
is \textit{the Calabi invariant of $\varphi$}, which is in turn independent of the primitive of $\omega$.

Now we are ready to give the scheme of the plug construction and show the connection between contact geometry of the plug and the classical invariants given above.

We consider  a smooth path $\{\varphi_t\}_{t\in [0,1]}\subset  {\rm Diff}_c(\mathbb{B},\omega)$ that connects the identity to the map $\varphi:=\varphi^1$ and  we pick $L>0$. We define the function 
$$\tau:=\sigma_{\varphi,\lambda}+L:\mathbb{B}\rightarrow \mathbb{R}.$$ 
It is clear from the definition (\ref{action_general}) that $\sigma_{\varphi,\lambda}$ vanishes near the boundary of $\mathbb{B}$ and therefore the minimum of $\tau$ on $\mathbb{B}$ exists. Now we assume that  
\begin{eqnarray}\label{tau_bigger_0}
\min_{\mathbb{B}} \tau>0.
\end{eqnarray}
Then the map $$g: \mathbb{R}\times \mathbb{B} \rightarrow \mathbb{R}\times \mathbb{B},\; (s,z)\mapsto (s-\tau(z),\varphi(z))$$
defines a free $\mathbb{Z}$-action on $ \mathbb{R}\times \mathbb{B}$. Hence we get the quotient manifold, denoted by $M$. It is clear that the 1-form $ds+\lambda$ is a contact form on $\mathbb{R}\times \mathbb{B}$ and it is invariant under the action of $g$, namely $g^*(ds+\lambda)=ds+\lambda$. Hence 
$ds+\lambda$ induces a contact form $\eta$ on $M$ determined by the equation $p^*\eta=ds+\lambda$ where 
$$p:\mathbb{R}\times \mathbb{B}\rightarrow M$$ 
is the quotient map. As $g$ maps the graph $\{(\tau(z),z)\,|\,z\in \mathbb{B}\}$
of $\tau$ onto $\{0\}\times \mathbb{B}$, the subset $M_0$ enclosed by 
these two hypersurfaces is a fundamental domain. By Fubini's theorem, we have
\begin{eqnarray*}
{\rm vol}(M,\eta)
&=&{\rm vol}(M_0,ds+\lambda)\\
&=&\int_{M_0} (ds\wedge \lambda)\wedge (ds\wedge \lambda)^m\\
&=&\int_{M_0} ds\wedge \omega^m\\
&=&\int_{\mathbb{B}} \tau\, \omega^m\\
&=&\int_{\mathbb{B}} (L+\sigma)\, \omega^m\\
&=&L\pi^m+\int_{\mathbb{B}} \sigma_{\varphi,\lambda}\, \omega^m\\
&=&L\pi^m+ {\rm CAL}(\varphi)
\end{eqnarray*}
Now let  $U$ be a small neighborhood of the boundary of $\mathbb{B}$, 
on which $\tau\equiv L$ and $\varphi={\rm id}$. 
Then $g(s,z)=(s-L,z)$ on $\mathbb{R}\times U$ and therefore $p(\mathbb{R}\times U)$ is diffeomorphic to 
$\mathbb{R}/L\mathbb{Z}\times U$ and $\eta$ is identified with $ds+\lambda$ on this subset. We note that with this identification, the Reeb vector field  $R_\eta$ associated to $\eta$ reads as $\partial_s$ and its flow is simply the translation along the coordinate $s$ on $\mathbb{R}/L\mathbb{Z}\times U$. Moreover, it is not hard to see that $p(\{0\}\times \mathbb{B})$ is a global hypersurface of sections for the flow of $R_\eta$ with the first return time map $\tau$ and first return map $\varphi$. We also note that for each $s\in \mathbb{R}$, the restriction $p_{|\{s\}\times \mathbb{B}}$ pulls back $d\eta$ to $\omega$.

After carrying out the above construction for all $\varphi^t$, we may view the contact manifold $(M,\eta)$ as a smooth deformation of the contact manifold $(\mathbb{R}/L\mathbb{Z}\times \mathbb{B},ds+\lambda)$, provided that 
\begin{eqnarray}\label{tau_t_bigger_0}
\min_{\mathbb{B}} \tau_t>0.
\end{eqnarray}
for all $t\in [0,1]$, where 
$$\tau_t:=\sigma_{\varphi^t,\lambda}+L.$$
Moreover since $\{\varphi^t\}$ is compactly supported, along this deformation, the contact forms coincide with $ds+\lambda$ on a neighbourhood of the boundary, which is identified with $\mathbb{R}/L\mathbb{Z}\times U$, where $U$ is a neighbourhood of $\partial \mathbb{B}$. 

Now we claim that by replacing the unit disc $\mathbb{D}$  with the unit ball in the proof of Proposition 3.1 in \cite{3sphere}, the manifold $M$ may be reparametrized as $\mathbb{R}/L\mathbb{Z}\times \mathbb{B}$ in such a way that the above properties of the contact form $\eta$ are inherited by a contact form $\beta$ on $\mathbb{R}/L\mathbb{Z}\times \mathbb{B}$. More concretely, we have the following statement. 
\begin{lem}\label{contact_mapping_torus}
 Let $\{\varphi_t\}_{t\in [0,1]}$ be an isotopy of compactly supported symplectomrphisms on the open unit disc $\mathbb{B}$ such that $\varphi^0={\rm id}$ and $\varphi^1=\varphi$. Let $L$ be a positive number such that for all $t\in[0,1]$
 $$\tau_t:=\sigma_{\varphi^t,\lambda}+L>0 $$
 on $\mathbb{B}$. Then there exists a contact form $\beta$ on $\mathbb{R}/L\mathbb{Z}\times \mathbb{B}$ such that the followings hold.
\begin{enumerate} [label=\textrm{(a\arabic*)}]
 \item \label{a1}${\rm vol}(\mathbb{R}/L\mathbb{Z}\times \mathbb{B},\beta)=L\pi^m+ {\rm CAL}(\varphi)$.
\item \label{a2}$\beta=ds+\lambda$ on a neighbourhood of the boundary of $\,\mathbb{R}/L\mathbb{Z}\times \mathbb{B}$. In particular,  $R_\beta=\partial_s$ near the boundary and its flow is globally defined.  
 \item \label{a3}The hypersurface $\{0\}\times \mathbb{B}$ is transverse to the flow of $R_\beta$ and any Reeb orbit intersects $\{0\}\times \mathbb{B}$ in future and past.
 \item \label{a4} After the identification $\{0\}\times \mathbb{B}\cong  \mathbb{B}$, $\varphi$ is the first return map and $\tau:=\tau_1$ is the first return time associated to the hypersurface 
 $\{0\}\times \mathbb{B}$. 
 \item \label{a5}$\beta$ is smoothly isotopic to $ds+\lambda$ through a path of contact forms that coincide with $ds+\lambda$ on a fixed neighbourhood of the boundary of $\mathbb{R}/L\mathbb{Z}\times \mathbb{B}$.
\end{enumerate}
\end{lem}

\section{Radial hamiltonians on the ball}
The outcome of the above discussion concerning our problem is the following. If one comes up a suitable symplectic isotopy of the ball, in the sense that the action of the isotopy satisfies (\ref{tau_t_bigger_0}), the dynamics of the time-one-map is suitably arranged and the action and the Calabi invariant are suitably bounded, then one has a contact form $\beta$ on $\mathbb{R}/L\mathbb{Z}\times \mathbb{B}$ with desired contact volume and action spectrum of the Reeb dynamics. In this section, we construct an isotopy, which leads to a contact form on $\mathbb{R}/L\mathbb{Z}\times \mathbb{B}$ with arbitrarly large systolic ratio, using radial hamiltonians on the ball. Again, $\mathbb{B}\subset \mathbb{R}$ denotes the open unit ball and $m\geq 2$. 
\begin{lem}\label{ball_map} Let $L>0$ be given. Then for any $\varepsilon>0$, there exists a positive integer $n$ and a compactly supported isotopy $\{\varphi_t\}_{t\in [0,1]}\subset {\rm Diff}_c(\mathbb{B},\omega)$ such that $\varphi^0={\rm id}$ and the followings hold.
\begin{enumerate}[label=\textrm{(b\arabic*)}]
    \item \label{b1} $\sigma_{\varphi^t,\lambda}\geq -L+L/n$ for all $t\in [0,1]$.
    \item \label{b2} ${\rm CAL}(\varphi^1)+L\pi^m\leq\varepsilon$.
    \item \label{b3} All the fixed points of $\varphi^1$ have non-negative action.
    \item \label{b4} All the periodic points of $\varphi^1$ that are not fixed points have period at least $n$. 
\end{enumerate}
\end{lem}
\begin{proof} As in \cite{3sphere},  we consider an autonomous radial hamiltonian 
$$H:\mathbb{B}\rightarrow \mathbb{R},\; H(z)=h(r^2)$$
on the unit ball such that $h$ is non-negative and vanishes near $r^2=1$. 
Here, $z=(z_1,\dots,z_m)$, $r^2=\sum_{i=1}^m r_i^2$, $z_i=r_ie^{i\theta_i}$ for $i=1,\ldots,m$. As before, we take the primitive
$$\lambda=\frac{1}{2}\sum_{i=1}^m x_idy_i-y_idx_i=\frac{1}{2}\sum_{i=1}^m r_i^2d\theta_i$$ 
on $\mathbb{B}$.
The standard symplectic form $\omega=d\lambda$ decomposes as $\omega=\sum_{i=1}^m \omega_i$, where $\omega_i:=p_i^*\omega_0$, $p_i(z)=z_i$ is the 
projection and $\omega_0$ is the standard symplectic form on $\mathbb{R}^2$.

A straightforward computation shows that the hamiltonian vector field of $H$ reads as
$$X_H=-2h'(r^2)\sum_{i=1}^m \partial_{\theta_i}$$ 
and the flow of $X_H$ is given by 
$$\varphi^t(z)=e^{-2h'(r^2)ti}z.$$
It is easy to check that the action $\sigma_{\varphi^t}$ of a compactly supported hamiltonian isotopy $\varphi^t$ associated to a time-dependent hamiltonian $H_t$ is given by 
\begin{eqnarray}\label{action_formula}
\sigma_{\varphi^t,\lambda}(z)
=\int_{\{z\mapsto\varphi^s(z)\}_{s\in [0,t]}}\lambda + \int_0^tH_s(\varphi^s(z))\,ds.
\end{eqnarray}

\textit{Notation:} From now on, we omit the subscript for the primitive in the notation of the action whenever the fixed primitive $\lambda$ is meant. 

Via (\ref{action_formula}) we compute the action of the time-one-map.
\begin{eqnarray*}
\sigma_{\varphi^1}(z)
&=&\int_{\{z\mapsto\varphi^t(z)\}}\lambda + \int_0^1H_t(\varphi^t(z))\,dt\\
&=&\int_0^1\lambda(X_H)\,dt + \int_0^1H(e^{-2h'(r^2)ti}z)\,dt\\
&=&\int_0^1 (\frac{1}{2}\sum_{i=1}^m r_i^2d\theta_i)(-2h'(r^2)\sum_{i=1}^m\partial_{\theta_i})\,dt 
+\int_0^1h(||e^{-2h'(r^2)ti}z||^2)\,dt\\
&=&h(r^2)-r^2h'(r^2).
\end{eqnarray*}

\emph{Notation:} We let $\Omega_{\mathbb{B}}$ and $\Omega_{S^{2m-1}}$ be the standard riemannian volume forms on the corresponding manifolds and we write $V(\cdot)$ for the volumes of these spaces (or their subsets) with respect to these volume forms.

We want to compute the Calabi invariant. To this end we fix a reparametrization of the unit ball as follows.
\begin{eqnarray}\label{daleth}
\daleth: (0,1)\times S^{2m-1}\rightarrow \mathbb{B},\; (r,z)\mapsto rz.
\end{eqnarray}
We have 
$$\daleth^*(\Omega_{\mathbb{B}})=\frac{1}{m!}(\daleth^*\omega)^m=r^{2m-1}dr\wedge \Omega_{S^{2m-1}}.$$
Hence 
\begin{eqnarray*}
 {\rm CAL}(\varphi^1)
 &=& \int_{\mathbb{B}} \sigma \,\omega^m\\
 &=& m!\int_{\mathbb{B}} \sigma \,\Omega_{\mathbb{B}}\\
 &=& m!\int_{(0,1)\times S^{2m-1}}(h(r^2)-r^2h'(r^2))\,r^{2m-1}dr\wedge \Omega_{S^{2m-1}}\\
 &=&m!V(S^{2m-1})\int_0^1 (r^{2m-1}h(r^2)-r^{2m+1}h'(r^2))\,dr\\
 &=&2m(m+1)\pi^m\int_0^1r^{2m-1}h(r^2)\,dr.
\end{eqnarray*}
We note that 
\begin{eqnarray}\label{CALham}
{\rm CAL}(\varphi^1)=(m+1) \int_{\mathbb{B}} H \,\omega^m.
\end{eqnarray}

Now given $\delta\in (0,1/2)$ small enough, we pick a function $\chi_\delta: [0,\infty)\rightarrow \mathbb{R}$
supported in $[0,1)$ satisfying
\begin{itemize}
 \item $\chi_\delta(s)=1-\delta-s $ on $[0,1-2\delta]$;
 \item $\max \{1-\delta-s,0\}\leq \chi_\delta(s)\leq \max \{(1-\delta)(1-s)\}$;
 \item $\chi'_\delta(s)=-1$ for $s\in [0,1-2\delta]$ and $\chi'_\delta(s)\in [-1,0]$ for all $s$; 
 \item $0\leq \chi_\delta(s)-s\chi'_\delta(s)\leq 1-\delta$.
\end{itemize}
We put 
$$H_+(z)=\frac{\pi}{n}\chi_\delta(r^2)$$
so that the hamiltonian flow $\varphi_+^t$ satisfies $$\varphi_+^t(z)=e^{\frac{i2\pi t}{n}}z\; \textrm{for}\; r< \sqrt{1-2\delta}$$
and $\varphi_+^t$ is a rotation by an angle less or equal to $2\pi/n$ on $\mathbb{B}$.
We have 
$$\sigma_{\varphi_+^t}(z)=\frac{t\pi}{n}\left(\chi_\delta(r^2)-r^2\chi'_\delta(r^2)\right)$$
\begin{eqnarray}\label{phi+action}
\Rightarrow\; 
0\leq \sigma_{\varphi_+^t}\leq \frac{t\pi}{n}(1-\delta)\leq \frac{\pi}{n}\;\;\textrm{for all}\;\;t\in[0,1]. 
\end{eqnarray}
Using (\ref{CALham}) and (\ref{phi+action}), we get
\begin{eqnarray}\label{phi+calabi}
\Rightarrow\; 0\leq {\rm CAL}(\varphi_+^1)\leq \frac{\pi}{n}(m+1)!V(\mathbb{B})=\frac{(m+1)\pi^{m+1}}{n}.
\end{eqnarray}

We identify  $\mathbb{B}\setminus\{0\}$ with $(0,1)\times S^{2m-1}$ via (\ref{daleth}) and in these coordinates 
$\varphi_+^t$ is just a rotation along the fibres of the Hopf fibration $S^{2m-1}\rightarrow \mathbb{C}P^{m-1}$.
We trivialize the Hopf fibration over $\mathbb{C}^{m-1}\subset\mathbb{C}P^{m-1}$ and get a codimension-zero embedding 
$$g: \mathbb{D}^*  \times \mathbb{C}^{m-1} \rightarrow \mathbb{B},$$
$$g(re^{i\theta},z)=\left(\frac{re^{i\theta}z}{\sqrt{1+||z||^2}},\frac{re^{i\theta}}{\sqrt{1+||z||^2}}\right)\in \mathbb{C}^{m-1}\times \mathbb{C}$$
where $\mathbb{D}^*$ is the punctured unit disc. 

We note that the image of $g$ is invariant under 
$\varphi_+^t$. We define the sectors 
$$S_k:=\left\{re^\theta\in \mathbb{D}^*\,|\, (k-1)\frac{2\pi}{n}<\theta< k\frac{2\pi}{n}\right\},\; k=1,...,n. $$
We note that $\varphi_+^t$ is constant on the second component of the product $\mathbb{D}^*  \times \mathbb{C}^{m-1}$
and on the first component, it behaves as the corresponding disc map in \cite{3sphere}. For each $k$, we define an open subset $U_k:=g(S_k\times \mathbb{C}^{m-1})$ in $\mathbb{B}^{2m}$.    

Let $\rho\in (0,1)$ and $n>0$ be given. We claim that there is a finite collection of non-overlapping closed balls in $U_1$ such that
\begin{itemize}
\item the ratio between total volume of the balls 
and the volume of $U_1$ is at least $\rho$ ;
\item  the radius of each ball is less or equal to $\pi/n$.
\end{itemize}
The existence of such a collection of balls follows from the fact that the open set $U_1$ is a pre-compact subset of $\mathbb{B}$ with finite volume. 


We have a free $\mathbb{Z}_n$-action on $\mathbb{B}\setminus \{0\}$, given by the rotation by the angle $2\pi/n$ along the 
Hopf fibration, which is isometric. The corresponding action on $\mathbb{D}^*\times \mathbb{C}^{m-1}$ is just the rotation along 
the disc, which permutes the sectors $S_k\times \mathbb{C}^{m-1}$. Hence the  $\mathbb{Z}_n$-action permutes $U_k$'s. Moreover the volumes of $U_k$'s are identical. Now we take the collection $(B_j:=B(z_j,r_j))_{j\in J}$ given by the union of $\mathbb{Z}_n$- orbits of the balls that we choose in $U_1$. Then we have 
$$\sum_{j\in J} V (B(z_j,r_j))\geq \rho \,V(\mathbb{B})$$
and $r_j\leq \pi/n$ for all $j\in J$.

Let $\epsilon\in (0,1/2)$ be small and $c>0$. For each $j\in J$, we consider
the hamiltonian
$$K_j:\mathbb{B}\rightarrow \mathbb{R},\; K_j(z)=-c\chi_\epsilon\left(\frac{||z-z_j||^2}{r_j^2}\right).$$
Let $\psi_j^t$ be the hamiltonian flow of $K_j$. Since each $K_j$ is supported in $B_j$, the flow $\psi_j^t$ and the 
action $\sigma_{\psi_j^t}$ are supported in $B_j$. We compute the action via the translated hamiltonian 
$\tilde{K}_j(z)=-c\chi_\epsilon\left(\frac{||z||^2}{r_j^2}\right)$. Let $\tilde{\psi}_j^t$ be the hamiltonian flow of $\tilde{K}_j$. 
We have
$$\sigma_{\tilde{\psi}_j^t}(z)=-ct\left(\chi_\epsilon\left(\frac{||z||^2}{r_j^2}\right)-
\frac{||z||^2}{r_j^2}\chi'_\epsilon\left(\frac{||z||^2}{r_j^2}\right)\right)$$
and on $B(z_j, r_j\sqrt{1-2\epsilon})$,
$$\sigma_{\tilde{\psi}_j^t}(z)=-ct(1-\epsilon).$$
Moreover on $B_j$,  
$$0\geq \sigma_{\tilde{\psi}_j^t}\geq  -ct(1-\epsilon)\geq -c$$
for all $t\in [0,1]$. We compute
\begin{eqnarray*}
{\rm CAL}(\tilde{\psi}_j^1)
&=&\int \sigma_{\tilde{\psi}_j^1}\,\omega^m\\ 
&\leq& m!\int_{B(z_j, r_j\sqrt{1-2\epsilon})} \sigma_{\tilde{\psi}_j^1}\, \Omega_{\mathbb{B}}\\
&=&-m!c(1-\epsilon)\,V(B(z_j, r_j\sqrt{1-2\epsilon}))\\
&\leq&-c\pi^m(1-2\epsilon)^{m+1}r_j^{2m}.
\end{eqnarray*}
We pick an area preserving diffeomorphism $\tau_j: \mathbb{B}\rightarrow \mathbb{B}$ such that 
$$\tau_j(z)=z+z_j\;\;{\rm on }\;\; B(0, r_j).$$ 
We have 
$$\tilde{\psi}_j^t=\tau^{-1}_j\circ \psi_j^t\circ \tau_j.$$
It follows from the formula (\ref{action_formula}) that
$$\sigma_{\psi_j^t,\lambda}\circ\tau_j =\sigma_{\tau^{-1}_j\circ \psi_j^t\circ \tau_j,\tau_j^*\lambda}
=\sigma_{\tilde{\psi}_j^t,\tau_j^*\lambda}$$
and
$${\rm CAL}(\psi_j^1)={\rm CAL}(\tilde{\psi}_j^1).$$
For a detailed account on the behaviour of the action and Calabi invariant under the conjugation, we refer to Lemma 2.2  and Proposition 2.4 in \cite{3sphere}. 

We put $z_j:=(u_1+iv_1,\dots,u_m+iv_m)$ and note that
$$\tau_j^*\lambda-\lambda=\frac{1}{2}\sum (u_jdy_j-v_jdx_j)\;\;{\rm on }\;\; B(0, r_j).$$
Then there is some function $\eta:\mathbb{B}\rightarrow \mathbb{R}$ such that 
$\tau_j^*\lambda-\lambda=d\eta$ and 
$$\eta(z)=\frac{1}{2}\sum_{i=1}^m (u_jy_j-v_jx_j)=\frac{1}{2}\,{\rm Im}\langle z,z_j\rangle_{\mathbb{C}} \;\;{\rm on }\;\; B(0, r_j).$$
Hence we get
$$\sigma_{\tilde{\psi}_j^t,\tau_j^*\lambda}=\sigma_{\tilde{\psi}_j^t,\lambda}+\eta\circ\tilde{\psi}_j^t-\eta.$$
The above equation follows essentially from the formula (\ref{action_formula}) and for the details, we refer to Lemma 2.2 in \cite{3sphere}.

Since $\sigma_{\tilde{\psi}_j^t}$ is supported on $B(0,r_j)$, we get 
$$\Rightarrow\; |\sigma_{\tilde{\psi}_j^t,\tau_j^*\lambda}-\sigma_{\tilde{\psi}_j^t,\lambda}|
\leq \sup_{B(0,r_j)} \eta-\inf_{B(0,r_j)} \eta\leq r_j||z_j||\leq r_j\leq \frac{\pi}{n}.$$
Hence we get
$$-c-\frac{\pi}{n}\leq \sigma_{\psi_j^t}\leq \frac{\pi}{n}.$$
Since the diffeomorphisms $\{\psi_j^t\}_{j\in J}$ have disjoint support, they commute pairwise. We denote their composition by 
$\varphi_-^t$, which is the hamiltonian flow of the hamiltonian $H_-:=\sum_{j\in J}K_j$. 
We note that the support of $\varphi_-^t$ is contained in the union of $B_j$'s and each $B_j$ is invariant under $\varphi_-^t$. 
Hence we have 
\begin{eqnarray}\label{phi-action}
-c-\frac{\pi}{n}\leq \sigma_{\varphi_-^t}\leq \frac{\pi}{n},
\end{eqnarray}
and 
\begin{eqnarray*}
{\rm CAL}(\varphi_-^1)&=&\sum_{j\in J} {\rm CAL}(\psi_j^1)\\
&\leq& \sum_{j\in J} -c\pi^m(1-2\epsilon)^{m+1}r_j^{2m}\\
&=&-c(1-2\epsilon)^{m+1}\sum_{j\in J}\pi^mr_j^{2m}\\
&=&-m!c(1-2\epsilon)^{m+1}\sum_{j\in J}V(B_j).
\end{eqnarray*}
Hence
\begin{eqnarray}\label{phi-calabi}
{\rm  CAL}(\varphi_-^1)\leq -c(1-2\epsilon)^{m+1}\rho\pi^m. 
\end{eqnarray}
Above in the first equation, we use the fact that for fixed a symplectic form, 
$${\rm CAL}: {\rm Diff}_c(\mathbb{B},\omega)\rightarrow \mathbb{R}$$
is a group homomorphism, for details see \cite{3sphere}, \cite{Mcduff}. Now we assume that $\delta$ is small enough so that $\bigcup_{j\in J} B_j\subset B(0,\sqrt{1-2\delta})$. We consider 
the isotopy 
$$\varphi^t=\varphi_+^t\circ \varphi_-^t$$
which is compactly supported by the choice of $ \delta$ and connects the identity map to $\varphi^1=\varphi_+^1\circ \varphi_-^1$. 
We note that the fixed points of $\varphi^1$ are the origin and the points forming a neighborhood  
of $\partial \mathbb{B}$, which is outside of $B(0,\sqrt{1-2\delta})$. As $\bigcup_{j\in J}  B_j\subset B(0,\sqrt{1-2\delta})$, 
$\varphi_+^1$ permutes $B_j$'s so that none of the fixed points of $\varphi_-^1$ remains as a fixed point. Similarly, any  
periodic point of $\varphi^1$ within  $B(0,\sqrt{1-2\delta})$ has period at least $n$ and any other periodic point of 
$\varphi^1$ is simply a periodic point of $\varphi_+^1$ and has period at least $n$.  
By the formula
$$\sigma_{\varphi^t}=\sigma_{\varphi^t_+}\circ\varphi^t_-+\sigma_{\varphi^t_-},$$
which is a straightforward generalizatin of Lemma 2.2 (ii) in  \cite{3sphere}, we get
$\sigma_{\varphi^1}(0)=\frac{\pi}{n}(1-\delta)$
for the fixed point $0$ and $\sigma_{\varphi^1}(z)\geq 0$ for any other fixed point $z$. Hence \ref{b3} and \ref{b4} are established. 

Using (\ref{phi+action}) and (\ref{phi-action}),
we get 
\begin{eqnarray}\label{phiaction}
 -c-\frac{\pi}{n}\leq \sigma_{\varphi^t}\leq \frac{2\pi}{n}
\end{eqnarray}
and by (\ref{phi+calabi}) and (\ref{phi-calabi}), we get
\begin{eqnarray}\label{phicalabi}
 {\rm CAL}(\varphi^1)\leq \frac{(m+1)\pi^{m+1}}{n}-c(1-2\epsilon)^{m+1}\rho\pi^m.
\end{eqnarray}
We put 
$$c=L-\frac{L+\pi}{n}$$
where $n$ is large enough and get
\begin{eqnarray}\label{phiactionL}
 -L+\frac{L}{n}\leq \sigma_{\varphi^t}\leq \frac{2\pi}{n}
\end{eqnarray}
for all $t\in [0,1]$. This establishes \ref{b1}. From (\ref{phicalabi}) we get
$$
{\rm  CAL}(\varphi^1)\leq -L\pi^m(1-2\epsilon)^{m+1}\rho+
 \frac{1}{n}\bigg[(m+1)\pi^{m+1}+(L+\pi)(1-2\epsilon)^{m+1}\rho\pi^m\bigg].
$$ 
Then we have 
\begin{eqnarray*}
{\rm CAL}(\varphi^1)+L\pi^m
&\leq &L\pi^{m}(1-(1-2\epsilon)^{m+1}\rho)\\
&&+\frac{1}{n}\bigg[(m+1)\pi^{m+1}+(L+\pi)(1-2\epsilon)^{m+1}\rho\pi^m\bigg]
\end{eqnarray*}
Now given $\varepsilon>0$,  we choose $\epsilon$ close enough to 0 and $\rho$ close enough to 1 so that
$$(1-(1-2\epsilon)^{m+1}\rho)\leq \frac{\varepsilon}{2L\pi^m}$$
and we choose $n$ large enough so that 
$$\frac{1}{n}\bigg[(m+1)\pi^{m+1}+(L+\pi)\pi^m\bigg]\leq \frac{\varepsilon}{2}.$$
With these choices we get
\begin{eqnarray*}
{\rm CAL}(\varphi^1)+L\pi^m
&\leq &L\pi^{m}(1-(1-2\epsilon)^{m+1}\rho)\\
&&+\frac{1}{n}\bigg[(m+1)\pi^{m+1}+(L+\pi)(1-2\epsilon)^{m+1}\rho\pi^m\bigg]\\
&\leq& L\pi^{m} \frac{\varepsilon}{2L\pi^m}+\frac{1}{n}\bigg[(m+1)\pi^{m+1}+(L+\pi)\pi^m\bigg]\\
&\leq &\frac{\varepsilon}{2}+\frac{\varepsilon}{2}=\varepsilon.
\end{eqnarray*}
This establishes \ref{b2}.
\end{proof}

\section{The plug and the systolic ratio on spheres}
In this section, we first combine the lemmata in the above sections in order to define the plug of any odd dimension (compare with Proposition 2 in \cite{general}). Next we prove that the contact systolic ratio on any sphere with the standard contact structure is unbounded. 
\begin{prop}\label{the_plug} Let $\epsilon>0$  and $L>0$ be given. Let $\lambda$ be the standard primitive of $\omega$ on the 2m-dimensional open unit ball $\mathbb{B}$. Then there exists a smooth contact form $\beta$ on $\mathbb{R}/L\mathbb{Z}\times \mathbb{B}$ such that the followings hold. 
\begin{enumerate} [label=\textrm{(p\arabic*)}]
 \item \label{p1}${\rm vol}(\mathbb{R}/L\mathbb{Z}\times \mathbb{B},\beta)\leq \epsilon$.
 \item \label{p2}$\beta=ds+\lambda$ on a neighbourhood of the boundary of $\,\mathbb{R}/L\mathbb{Z}\times \mathbb{B}$. In particular,  $R_\beta=\partial_s$ near the boundary and its flow is globally defined.
 \item \label{3} All closed orbits of $R_\beta$ have period at least $L$.  
 \item \label{p4}$\beta$ is smoothly isotopic to $ds+\lambda$ through a path of contact forms that coincide with $ds+\lambda$ on a fixed neighbourhood of the boundary of $\,\mathbb{R}/L\mathbb{Z}\times \mathbb{B}$.
\end{enumerate}
\end{prop}
\begin{proof}
We take the isotopy given in Lemma \ref{ball_map} and apply the Lemma \ref{contact_mapping_torus} to get the smooth contact form $\beta$. Combining \ref{b2} with \ref{a1} leads to \ref{p1}. The statement \ref{p2} and \ref{p4} are given by \ref{a2} and \ref{a5} respectively. 

By \ref{a4}, we know that a periodic orbit of $R_\beta$ corresponds to either a fixed point of $\varphi^1$ or to a periodic point of $\varphi^1$ that is not a fixed point. In the first case, if $z\in \mathbb{B}$ is the corresponding fixed point, then  the period is given by  
$$\tau(z)=\sigma_{\varphi^1}(z)+L\geq L$$ 
since $\sigma_{\varphi^1}(z)\geq 0$ by \ref{b3}. We note that by \ref{b1}, we have the global bound 
$$\tau=\sigma_{\varphi^1}+L\geq L/n.$$ 
Hence in the second case, if $z$ is the corresponding periodic point with period $k$, the period of the Reeb orbit is given by 
$$\sum _{j=0}^{k-1}\tau(\varphi^j(z))$$
where $\varphi^j=(\varphi^1)^j$ and by \ref{b4}, one has $k\geq n$ and therefore
$$\sum _{j=0}^{k-1}\tau(\varphi^j(z)) \geq k \frac{L}{n}\geq L.$$  
\end{proof}
In order to provide the first application of the plug construction we recall the standard contact structure on an odd sphere. 
For $m\geq 2$, we consider \emph{the standard contact form} $\alpha_0$ on the unit sphere $S^{2m+1}$ given by
$$\alpha_0=\frac{1}{2}\sum_{i=1}^{m+1} (u_idv_i-v_idu_i)$$
via the coordinates $z=(z_1,...,z_{m+1})\in S^{2m+1}\subset\mathbb{C}^{m+1}$, $z_j=u_j+iv_j$. 
The contact structure $\xi_{{\rm st}}:=\ker \alpha_0$ is called \textit{the standard contact structure}. The Reeb vector field $R_{\alpha_0}$ generates the Hopf fibration $\pi: S^{2m+1}\rightarrow \mathbb{C}P^m$ so that each 
closed orbit has period $\pi$. We also know that ${\rm vol}( S^{2m+1},\alpha_0)=\pi^{m+1}$, therefore
 $$\rho(S^{2m+1},\alpha_0)=\frac{(T_{{\rm min}}(\alpha_0))^{m+1}}{{\rm vol}(S^{2m+1},\alpha_0)}=\frac{\pi^{m+1}}{\pi^{m+1}}=1.$$
\begin{thm}\label{bigspheres} Given any $C>0$, there is a contact form $\alpha$ on $S^{2m+1}$ such that $\ker \alpha=\xi_{{\rm st}}$ and 
$$\rho(S^{2m+1},\alpha)\geq C.$$
\end{thm}
\begin{proof}
Using the embedding 
$$\mathbb{C}^m\hookrightarrow\mathbb{C}P^m,\; z=(z_1,\dots,z_m)\mapsto (z_1:\ldots:z_m:1)$$
we trivialize the Hopf bundle as
$$\mathbb{R}/\pi\mathbb{Z}\times \mathbb{C}^n\rightarrow S^{2m+1},\; (s,z)
\mapsto \kappa(z)(e^{i2s}z,e^{i2s})\in \mathbb{C}^{m+1}$$
where $\kappa(z)=(1+||z||^2)^{-1/2}$. 
Via the rescaling $z\mapsto \frac{z}{\sqrt{1-||z||^2}}$, we define a diffeomorphism
\begin{eqnarray}\label{model}
f: \mathbb{R}/\pi\mathbb{Z}\times \mathbb{B}\rightarrow S^{2m+1},\;f(s,z)=(e^{i2s}z,e^{i2s}\sqrt{1-||z||^2}).
\end{eqnarray}
where as before, $\mathbb{B}$ is the open unit ball in $\mathbb{R}^{2m}$. On $\mathbb{B}$, we fix the primitive 
$$\lambda=\frac{1}{2}\sum_{i=1}^m (x_idy_i-y_idx_i).$$
A straightforward computation shows  that
 $$f^*\alpha_0=ds+\lambda.$$
We note that $S^{2m+1}\setminus\, {\rm im}(f)=\pi^{-1}(\mathbb{C}P^{m-1})$ where 
 $$\mathbb{C}P^{m-1}\hookrightarrow \mathbb{C}P^{m},\; (z_0:\ldots:z_{m-1})\mapsto (z_0:\ldots:z_{m-1}:0).$$
That is, $f$ misses a subset of codimension two and hence of measure zero. 

Let $C>0$ be given. We take $\epsilon=\pi^{m+1}/C$ and $L=\pi$. By Proposition \ref{the_plug}, there is a contact form $\beta$ on $\mathbb{R}/\pi\mathbb{Z}\times \mathbb{B}$ satisfying \ref{p1}-\ref{p4}. By \ref{p4}, the push forward of $\beta$ by $f$, still denoted by $\beta$, extends to a smooth contact form on $S^{2m+1}$ such that the minimal period is $\pi$. We also note that 
$${\rm vol}(S^{2m+1},\beta)={\rm vol}(\mathbb{R}/\pi\mathbb{Z}\times \mathbb{B}^{2m},\beta)\leq \epsilon. $$
Hence we get 
$$\rho(S^{2m+1},\beta)\geq \frac{\pi^{m+1}}{\epsilon}=C.$$
We finally note that the contact form $\beta$ on $S^{m+1}$ is isotopic to $\alpha_0$. Hence $\ker \beta$ is isotopic to $\xi_{{\rm st}}$. By Gray's stability theorem, there is a diffeomorphism $\psi$ of $S^{2m+1}$ such that $\psi_*\xi_{{\rm st}}=\ker \beta$. Since systolic ratio is invariant under a diffeomorphism, $\alpha:=\psi^*\beta$ is the desired contact form. 
\end{proof}

\section{Contact forms with arbitrarily large systolic ratio}
In this section we prove that any co-orientable contact manifold $(V,\xi)$ admits a 
contact form with arbitrarly large systolic ratio. Our statement generalizes the main theorem of \cite{general} and our proof uses the  strategy of \cite{general} with necessary modifications. 
\subsection{The plug with arbitrary radius and primitive}
We first note that the plug construction given above, generalizes to the balls of arbitrary radius. In fact, given $r>0$ and $\varphi\in {\rm Diff}_c(\mathbb{B},\omega)$, we define a diffeomorphism 
$$\varphi_r:r\mathbb{B}\rightarrow r\mathbb{B},\;\varphi_r(z)=r\varphi(\frac{z}{r}),$$
which is compactly supported and symplectic. The definitions (\ref{action_general}) and (\ref{calabi_general}) make sense for $\varphi_r$ and a straightforward computation shows that  
\begin{equation} \label{action_r}
\sigma_{\varphi_r}=r^2\sigma_\varphi(\frac{z}{r})
\end{equation}
and 
\begin{equation}\label{calabi_r}
{\rm CAL}(\varphi_r)=r^{2m+2}{\rm CAL}(\varphi).
\end{equation}
We recover Proposition \ref{the_plug} for the ball of arbitrary radius as follows. Given $L>0$ and $\epsilon>0$, by Lemma \ref{ball_map}, one gets an isotopy $\{\varphi^t\}\subset {\rm Diff}_c(\mathbb{B},\omega)$ such that 
\begin{equation*}
\sigma_{\varphi^t,\lambda}\geq -\frac{L}{r^2}+\frac{L}{r^2 n}\;\; \forall t\in [0,1].
\end{equation*}
and 
\begin{equation*}
{\rm CAL}(\varphi^1)+\frac{L}{r^2}\pi^m<\frac{\epsilon}{r^{2m+2}}.
\end{equation*}
Combining these statements with (\ref{action_r}) and (\ref{calabi_r}), we get  
\begin{eqnarray}\label{b1_r}
\sigma_{\varphi_r^t,\lambda}= r^2 \sigma_{\varphi^t,\lambda}\geq -r^2\frac{L}{r^2}+r^2\frac{L}{r^2 n}=-L+L/n
\end{eqnarray}
and 
\begin{equation}\label{b2_r}
{\rm CAL}(\varphi_r^1)+L\pi^m r^{2m}=r^{2m+2}{\rm CAL}(\varphi^1)+r^{2m+2}\frac{L}{r^2}\pi^m < r^{2m+2}\frac{\epsilon}{r^{2m+2}}=\epsilon.
\end{equation}
Finally we note that repeating the construction in Lemma \ref{contact_mapping_torus} with the isotopy $\{\varphi_r^t\}\subset{\rm  Diff}_c(r\mathbb{B},\omega)$, we get a contact form $\beta$ on $\mathbb{R}/L\mathbb{Z}\times r\mathbb{B}$ so that the property \ref{a1} is replaced by
\begin{equation}\label{a1_r}
{\rm vol}(\mathbb{R}/L\mathbb{Z}\times r\mathbb{B},\beta)=L\pi^mr^{2m}+ {\rm CAL}(\varphi_r^1).
\end{equation}
Hence combining (\ref{b2_r}) with (\ref{a1_r}) yields 
$${\rm vol}(\mathbb{R}/L\mathbb{Z}\times r\mathbb{B},\beta)\leq \epsilon.$$
We note that the statements  \ref{b3} and \ref{b4} hold for the action $\sigma_{\varphi_r^1}$. Hence the proof of Proposition \ref{the_plug} goes through.

The final modification of the plug construction is replacing the standard primitive  $\lambda$ with an arbitrary primitive. After the observations made above, the proof of the following statement is a straightforward generalization of Proposition 2 in \cite{general}. 
\begin{lem}\label{plug_any_primitive} Let $r,\epsilon>0$ be given and let $\lambda'$ be any primitive of the standard symplectic 
form $\omega$ on the open disc $r\mathbb{B}$. Then there exists a smooth contact form $\beta'$ on 
$r\mathbb{B}\times \mathbb{R}/\mathbb{Z}$ such that the followings hold.
\begin{enumerate}[label=(p'\arabic*)]
 \item \label{p'1}$\beta'=\lambda'+ds$ in a neighborhood of  $\partial (r\mathbb{B}\times \mathbb{R}/\mathbb{Z}$).
 \item \label{p'2}$\beta'$ is smoothly isotopic to the contact form $\lambda'+ds$ on 
 $r\mathbb{B}\times \mathbb{R}/\mathbb{Z}$ through a path of contact forms which agree with $\lambda'+ds$
 in a neighborhood of $\partial(r\mathbb{B}\times \mathbb{R}/\mathbb{Z})$.
 \item \label{p'3} All the closed orbits of $R_{\beta'}$ have period at least 1.
 \item ${\rm vol}(r\mathbb{B}\times \mathbb{R}/\mathbb{Z},\beta')<\epsilon$.
\end{enumerate}
\end{lem}
\begin{proof} From the observations made above, we know that such a contact form $\beta$ exists when we replace $\lambda'$
 with the standard primitive $\lambda$ and take $L=1$. 
 
The 1-form $\lambda'-\lambda$ is closed on $r\mathbb{B}$ and hence exact. Let $u$ be a smooth function on
 $r\mathbb{B}$ so that $\lambda'-\lambda=du$. 
 
 Let $\beta_t$ be a smooth path of contact forms on $r\mathbb{B}\times \mathbb{R}/\mathbb{Z}$ such that 
 $\beta_0=\beta$, $\beta_1=\lambda+ds$ and $\beta_t=\lambda+ds$ on 
 $(r\mathbb{B}\setminus r'\mathbb{B})\times \mathbb{R}/\mathbb{Z}$ for all $t$ where $r'<r$. Let 
 $\chi$ be a smooth function on $r\mathbb{B}$ such that $\chi=0$ on $r'\mathbb{B}$ and 
 $\chi=1$ on $r\mathbb{B}\setminus r''\mathbb{B}$ where $r'<r''<r$. We define the contact form
 $$\beta':=\beta+d(\chi u).$$
 In fact, 
 $$\beta'\wedge (d\beta')^n=\beta'\wedge (d\beta)^n=\beta\wedge (d\beta)^n+d(\chi u)\wedge (d\beta)^n
 =\beta\wedge (d\beta)^n$$
 since on the region where $d(\chi u)$ is supported, $d\beta=d\lambda$ and $\partial_s$ is in the kernel of  
 both $d(\chi u)$ and $d\beta$. In particular, 
 $${\rm vol}(r\mathbb{B}\times \mathbb{R}/\mathbb{Z},\beta')={\rm vol}(r\mathbb{B}\times \mathbb{R}/\mathbb{Z},\beta)<\epsilon.$$
 We also note that on the region $(r''\mathbb{B}\setminus r'\mathbb{B})\times \mathbb{R}/\mathbb{Z}$, $R_\beta=\partial_s$ and hence
 $$\beta'(R_\beta)=\beta(R_\beta)+d(\chi u)(\partial_s)=1.$$
 Therefore, $\beta'(R_\beta)=1$  and $\imath_{R_\beta}d\beta'=0$ on $r\mathbb{B}\times \mathbb{R}/\mathbb{Z}$, that 
 is $R_{\beta'}=R_\beta$ and the third property is satisfied. 
 We note that on the region $(r\mathbb{B}\setminus r''\mathbb{B})\times \mathbb{R}/\mathbb{Z}$, 
 $$\beta'=\beta+du=\lambda'+ds$$
 and hence first property is satisfied. 
 We finally define
 $$\tilde{\beta}^t:=\beta^t+d(\chi u).$$
 We note that on the region $(r''\mathbb{B}\setminus r'\mathbb{B})\times \mathbb{R}/\mathbb{Z}$, 
 $$\tilde{\beta}^t=\lambda+ds\Rightarrow d\tilde{\beta}^t=d\lambda$$ 
 and therefore $\tilde{\beta}^t$ is a smooth homotopy of contact forms so that 
 $$\tilde{\beta}^0=\beta^0+d(\chi u)=\beta+d(\chi u)=\beta'$$
and 
$$\tilde{\beta}^1=\beta^1+d(\chi u)=\lambda+ds+d(\chi u)$$
and on the region $(r\mathbb{B}\setminus r''\mathbb{B})\times \mathbb{R}/\mathbb{Z}$, $\tilde{\beta}^t=\lambda+ds$ for all $t$. Next we consider the smooth 
isotopy 
$$\lambda+ds+tdu+(1-t)d(\chi u)$$
of contact forms that connects $\lambda+ds+d(\chi u)$ to $\lambda'+ds$, which clearly coincides with 
$\lambda'+ds$ on the region $(r\mathbb{B}\setminus r''\mathbb{B})\times \mathbb{R}/\mathbb{Z}$ for all $t$. Hence the second property follows.
\end{proof} 
\subsection{Construction of a special contact form}
The next step is to construct a closed contact form on a given contact manifold $(V,\xi)$ such that the Reeb flow gives rise to a circle bundle structure on a "large" portion of $V$. Such a contact form may be modified by inserting plugs so that most of the contact volume is sucked up without decreasing the minimal period. To this end, we first summarize the definitions and results concerning the Giroux's correspondence between the contact structures and supported open books in higher dimensions. For the details, we refer to \cite{openbook} and \cite{ILD}.

Let $F$ be a $2n$-dimensional domain with boundary $K$ and let $F^o$ denote the interior of $F$. A symplectic form $\omega\in \Omega^2(F^o)$ is called an \emph{ideal Liouville structure}, abbreviated by ILS,  on $F$ if it admits a primitive $\lambda \in \Omega^1(F^o)$ such that for some/any smooth function 
\begin{equation}\label{prop_of_u}
u:F\rightarrow [0,+\infty),\;\; {\rm where}\;\; K=u^{-1}(0)\;\; \textrm {is a regular level set,} 
\end{equation}
the 1-form $u\lambda$ on $F^o$ extends to a smooth 1-form $\beta$ on $F$, which is a contact form along $K$. 

If such a 2-form $\omega$ exists, then the pair $(F,\omega)$ is called an\emph{ ideal Liouville domain}, abbriviated as ILD, and any primitive $\lambda$ of above property is called an \emph{ideal Liouville form}, abbriviated as ILF. It turns out that given an ILD $(F,\omega)$, the contact structure 
$$\xi:=\ker (\beta|_{TK})$$
depends on the 2-form $\omega$ but not on $\lambda$ or $u$, see Proposition 2 in \cite{ILD}. Moreover, once $\lambda$ is chosen, one can recover all possible (positive) contact forms on $(K,\xi)$ by restriction of the extension of $u\lambda$ to $K$ as $u$ moves among the functions with the property (\ref{prop_of_u}). Hence the pair $(K,\xi)$ is called the \emph{ideal contact bounday} of $(F,\omega)$. We note that the orientation of $K$ that is determined by the co-oriented contact structure $\xi$ coincides with the orientation of $K$ as the boundary of $(F,\omega)$. 

A very useful feature of an ILD is that the vicinity of its bounday admits an explicit parametrization by means of which any ILF has a very nice form.  
\begin{lem}\label{nearKlemma} Let $(F,\omega)$ be an ILD and $\lambda$ be an ILF. Let $u$ be a function satisfying (\ref{prop_of_u}) and $\beta$ be the extension of $u\lambda$. Then for any  contact form $\alpha_0$ on $(K,\xi)$, there exist a constant $R>0$ and an embedding 
\begin{eqnarray*}
\imath:[0,R]\times K \rightarrow F
\end{eqnarray*}
such that 
$$\imath^*\lambda=\frac{1}{r}\alpha_0\: \textrm{ and } \:\imath(0,q)=q\: \textrm{ for all } \:q\in K,$$ 
where $r\in [0,R]$. In paricular, 
$$\imath^*\beta=\frac{u\circ\imath}{r}\alpha_0 \:\textrm{ on }\:F^o$$ and for all $q\in K$
$$(\beta_{|TK})(q)=\left(\frac{\partial (\imath\circ u)}{\partial r}(0,q)\right)\alpha_0.$$
\end{lem} 
\begin{proof} The above statement is a reformulation of Proposition 3 in \cite{ILD}. Here we give a proof that is essentially the same as the proof of that proposition but it is more explicit. 

Let $\dim F=2n$. As simple computation gives
$$\omega^n=\left(d(\beta/u)\right)^n=u^{-n-1}(ud\beta+n\beta\wedge du)\wedge (d\beta)^{n-1}.$$
We put 
$$\mu:=(ud\beta+n\beta\wedge du)\wedge (d\beta)^{n-1}$$
and note that by definition, $\mu$ is a smooth positive volume form on $F^o$. Since $u=0$ on $K$, 
$$\mu=-ndu\wedge \beta \wedge (d\beta)^{n-1}\;\;{\rm on}\;\;K$$
and it is a positive volume form since $\beta$ is a positive contact form along $K$ and $u$ satisfies (\ref{prop_of_u}). We define  the vector field $X$ on $F$ via 
$$\iota_X\mu=n\beta \wedge (d\beta)^{n-1}.$$
Evaluating the above identity at any point on $K$ shows that $\iota_Xdu=-1$ on $K$, that is, $X$ is non-singular on $K$ and points transversely outwards. On the other hand, writing $\beta=u\lambda$ on $F^0$ yields
$$\iota_X\mu=n\beta \wedge (d\beta)^{n-1}=nu^n\lambda \wedge (d\lambda)^{n-1}=u^n\iota_Y\omega^n=u^{-1}\iota_Y\mu,$$
where $Y$ is the Liouville vector field of $\lambda$ on $F^o$. In particular, $Y=uX$ on $F^o$. We note that the 1-form $\beta$ is not in general a Liouville form but here, one should think of the vector 
field $X$ as the "Liouville vector field of $\beta$". In fact, 
$$L_X\beta=\rho \beta$$
for some smooth function $\rho:F\rightarrow \mathbb{R}$. To see this, we note that on $F^o$,
$$L_X\beta=\iota_X d\beta+d(\iota_X\beta)=\frac{1}{u}\iota_Y(du\wedge\lambda+ud\lambda)=\frac{1}{u}\iota_Y du+\lambda=\frac{1}{u}(\iota_X du+1)\beta.$$ 
We claim that the fuction 
$$\rho:=\frac{1}{u}(\iota_X du+1)$$
extends to a smooth function on $F$. Let $\beta_0$ be the contact form, given by the restriction of $\beta$ to $TK$ and let $R_0$ be the Reeb vector field on $K$ associated to  $\beta_0$. Near $K$, we define  a smooth vector field $R$ via
$$R(p):=(\varphi^t_*R_0)(p);\; \varphi^t(q)=p,\; t\leq 0$$
where $\varphi^t$ is the flow of $X$. Since $R=R_0$ on $K$, $\beta(R)=1$ on $K$ and therefore on a compact neighbourhood $V$ of $K$, we have $\beta(R)>0$. On $V\setminus K$,
$$\rho=\frac{(L_X\beta) (R)}{\beta(R)}$$
and the right hand side is a smooth function on $V$. Hence $\rho$ extends to a smooth function on $V$ and $L_X\beta=\rho\beta$ on $F$.
In particular, the flow $\varphi^t$ preserves the kernel of $\beta$. Integrating the equation 
$$\frac{d}{dt}(\varphi^t)^*\beta=(\varphi^t)^*(L_X\beta)=(\rho\circ \varphi^t)(\varphi^t)^*\beta,$$
we get
$$(\varphi^t)^*\beta=\mu_t\beta,\; \mu_t=\exp \left(\int_0^t\rho\circ \varphi^s\,ds\right).$$
We also note that on $F^0$, 
$$\beta(X)=u\lambda(X)=\lambda(Y)=0.$$
Hence $\beta(X)=0$ on $F$. 
Now we consider the smooth embedding 
$$\Phi:(-\infty,0]\times K\rightarrow F,\;(t,q)\mapsto \varphi^t(q),$$
where $\varphi^t$ is the flow of $X$. By definition, we have $\Phi^*X=\partial_t$. We put $\hat{\beta}:=\Phi^*\beta$, $\hat{u}:=\Phi^*u$ and $\hat{\rho}:=\Phi^*\rho$. The above discussion says that $\hat{\beta}$ has no $dt$ component. Moreover, it has the form 
$$\hat{\beta}(t,q)=\exp\left(\int_0^t\hat{\rho}(s,q)\,ds\right)\beta_0(q),$$
where 
$$\hat{\rho}(t,q)=\frac{\hat{u}_t(t,q)+1}{\hat{u}(t,q)}.$$
Now let $\alpha_0$ be a positive contact form on $K$. Then there is some positive function $\kappa$ on  $K$ so that $\beta_0=\kappa\alpha_0$. We put $\hat{\lambda}:=\Phi^*\lambda$ and get  
$$\hat{\lambda}(t,q)=\Lambda(t,q)\alpha_0(q),$$
where $\Lambda: (-\infty,0)\times K\rightarrow (0,+\infty)$ is the smooth function given by
$$\Lambda(t,q):=\frac{1}{\hat{u}(t,q)}\kappa(q) \exp\left(\int_0^t\hat{\rho}(s,q)\,ds\right).$$
We note that 
$$\frac{\partial \Lambda}{\partial t}
=\kappa \exp\left(\int_0^t\hat{\rho}(s,\cdot)\,ds\right)\frac{\hat{\rho}\hat{u}-\hat{u}_t}{\hat{u}^2}=\frac{\Lambda}{\hat{u}}$$
and get 
$$\Lambda(t,q)=\Lambda(-1,q)\exp\left(\int_{-1}^t\frac{1}{\hat{u}(s,q)}\,ds\right).$$
We note that for each $q\in K$, the image of the map $\Lambda(\cdot,q)$ is an interval of the form $(c,\infty)$  since  $\hat{u}_t(0,q)=-1$. Since $K$ is compact, there exists some $R>0$ such that for all $q\in K$ there is some $c_q<0$ such that $\Lambda(\cdot, q)$ is a diffeomorphism between $[c_q,0)$ and $[1/R,+\infty)$ with non-vanishing derivative. We define a function 
$$f:(0,R]\times K\rightarrow (-\infty,0)$$
by imposing 
$$\Lambda(f(r,q),q)=\frac{1}{r}.$$ 
It is clear that $f$ is also smooth along $K$. Moreover it extends to a smooth function on $[0,R]\times K$ so that $f(0,q)=0$ for all $q$. We define a smooth embedding
$$F:[0,R]\times K\rightarrow (-\infty,0]\times K,\; F(r,q)=(f(r,q),q).$$
Then by definition 
$$F^*\hat{\lambda}=\frac{1}{r}\alpha_0\;\;{\rm and }\;\; F(0,q)=(0,q)$$
for all $q\in K$. Then $\imath:=\Phi\circ F$ is the desired embedding. The rest of the claim follows immediately form the equation $\imath^*\lambda=(1/r)\alpha$.
\end{proof}
From the above lemma, it follows that given two ILF's $\lambda_1$ and $\lambda_2$ on a an ILD $(F,\omega)$, there are two neighbourhoods $U_1$ and $U_2$ of $K:=\partial F$ such that there is an exact symplectomorphism between $(U_1\setminus K ,\lambda_1)$ and  
$(U_2\setminus K ,\lambda_2)$, which extends to $K$ as identity. Using this fact and the standard Moser argument, one shows the following.
\begin{prop}\label{cor5}(Corollary 5 in \cite{ILD})
Let $(F,\omega)$ be an ILD and $(\lambda_t)_{t\in [0,1]}$ 
be a path of ILF's on F, which is smooth in the sense that there exists a smooth path of 1-forms $(\beta_t)_{t\in [0,1]}$ on $F$ and some function $u$ satisfying (\ref{prop_of_u}) such that $\lambda_t=\beta_t/u$ for all $t\in [0,1]$. Then there is a symplectic
isotopy  $(\psi_t)_{t\in [0,1]}$ of F, relative to the boundary, such that  $\psi_0 ={\rm id}$ and, for
every $t\in [0, 1]$ the form 
$\psi^*\lambda_t-\lambda_0$ is the differential of a function with compact
support in $F^o$.
\end{prop}
Given an ILD $(F,\omega)$, it is not hard to see that the set of ILF's is an affine space. Hence any two IDF's can be connected by a smooth path of ILF's. Another consequence of Lemma \ref{nearKlemma} is the following stability result.  
\begin{prop}\label{lem6}(Lemma 6 in \cite{ILD})
Let F be a domain and $(\omega_t)_{t \in [0, 1]}$  be a smooth path of ideal
Liouville structures on F, in the sense that there exists a smooth path of 1-forms $(\beta_t)_{t\in [0,1]}$ on $F$ and some function $u$ satisfying (\ref{prop_of_u}) such that $\omega_t=d(\beta_t/u)$. Then there exists an isotopy $(\phi_t)_{t\in [0,1]}$ of $F$ such
that $\phi_0 ={\rm id}$ and $\phi_t^*\omega_t=\omega_0$ for all $t \in [0, 1]$. Moreover, we can choose this
isotopy relative to $K = \partial F$ if and only if  all forms $\omega_t$ induce the
same boundary contact structure.
\end{prop}

Ideal Liouville domains are particularly useful for clarifying the existence and uniqueness of the contact structures supported by open books in higher dimensions. We first recollect some facts on open books. 

An \textit{open book} in a closed manifold $V$ is a pair $(K,\Theta)$ where
\begin{enumerate}[label=\text{(ob\arabic*)}]
\item \label{ob1}$K\subset V$ is a closed co-dimension two submanifold with trivial normal bundle;
\item \label{ob2}$\Theta: V\setminus K\rightarrow S^1=\mathbb{R}/2\pi\mathbb{Z}$ is a locally trivial fibration such that $K$ has a neighbourhood $U$, which admits a parametrization $(re^{ix},q)\in \mathbb{D}\times K\cong U$ so that $\Theta$ reads as $\Theta(re^{ix},q)=x$ on $U$.
\end{enumerate}  
The submanifold $K$ is called the \textit{binding} of the open book and the closures of the
fibres of $\Theta$ are called the \textit{pages}. All the pages are compact manifolds, for which the binding is the common boundary. We note that the canonical orientation of $S^1$ induces co-orientations on the pages and the binding. Hence if $V$ is oriented then so are the pages and on the binding. Another way of defining an open book is the following. Let $h:V\rightarrow \mathbb{C}$ be a smooth function such that 
\begin{enumerate}[label=\text{(df\arabic*)}]
\item \label{df1}$h$ vanishes transversely;
\item \label{df2}$\Theta:=h/|h|:V\setminus K\rightarrow S^1$ has no critical points, where $K:=h^{-1}(0)$.
\end{enumerate}
Then the pair $(K,\Theta)$ is an open book in $V$. Moreover, any open book in $V$ may be recovered via a \emph{defining function} $h$ as above and such a defining function is unique up to multiplication by a positive function on $V$.

Given an open book $(K,\Theta)$ in a closed manifold $V$, one finds a vector filed $X$, refered as a \emph{spinning vector field}, on $V$ such that 
\begin{enumerate}[label=\text{(m\arabic*)}]
\item \label{m1}$X$ lifts to a smooth vector field on the manifold with boundary obtained from $V$ by a real oriented blow-up along $K$;
\item \label{m2}$X=0$ on $K$ and $(\Theta^*dx)(X)=2\pi$ on $V\setminus K$.
\end{enumerate}
Then the time-one-map of the flow of $X$ is a diffeomorphism  
$$\phi: F\rightarrow F$$ 
of the $0$th-page $F:=\Theta^{-1}(0)\cup K$, which fixes $K$. The isotopy class $[\phi]$ is called the \emph{monodromy} of the open book and it turns out that the open book is characterized by the pair $(F,[\phi])$. Namely, given the pair $(F,\phi)$, one defines the mapping torus 
$$MT(F,\phi):=([0,2\pi]\times F )\big/\sim \,;\; (2\pi,q)\sim (0,\phi(q)),$$
which is a manifold with boundary. One has the natural fibration 
$$\hat{\Theta}:MT(F,\phi)\rightarrow S^1,$$
where all fibres are diffeomorphic to $F$ and there is a natural parametrization of the fibre $\hat{\Theta}^{-1}(0)$ via the restriction of the above quotient map to $\{0\}\times F$. It turns out that if $\phi'\in [\phi]$, then there is a diffeomorphism between $MT(F,\phi)$ and $MT(F,\phi')$ that respects the fibrations over $S^1$ and the natural parametrizations of the $0$-th pages. Now given $MT(F,\phi)$, one collapses its boundary, which is diffeomorphic to $S^1\times K$, to $K$ and obtains so called the \emph{abstract open book} $OB(F,\phi)$. In fact, the closed manifold $OB(F,\phi)$ admits an open book given by the pair $(K,\Theta) $ where $\Theta$ is induced from $\hat{\Theta}$. Moreover, for $\phi'\in [\phi]$, the diffeomorphism between $MT(F,\phi)$ and $MT(F,\phi')$ descends to a diffeomorphism between corresponding abstract open books. In particular, $V$ and $OB(F,\phi)$ may be identified together with their open book structures. We note that one may choose a vector field $X$ that is actually smooth on $V$ (compare with \ref{m1}) and even 1-periodic near $K$. But it is not possible to obtain any given representative of the monodromy class via such a vector field. In fact, in order to obtain all representatives of the monodromy class, one needs to sweep out the whole affine space of spinning vector fields.

Open books meet with the contact topology via the following definition. Let $V$ be a closed manifold and $\xi$ be a co-oriented contact structure on $V$. We say $\xi$ is \emph{supported} by an open book $(K,\Theta)$ on $V$ if there is a contact form $\alpha$ on $(V,\xi)$, that is $\xi=\ker \alpha$, such that 
\begin{itemize}
\item $\alpha$ restricts to a (positive) contact form on $K$;
\item $d\alpha$ restricts to a (positive) symplectic form on each fibre of $\Theta$.
\end{itemize}
It turns out that given a closed contact manifold $V$, the isotopy classes of co-oriented contact structures are in one-to-one correspondence of (equivalence classes of) supporting open books. This statement is a very rough summary of what is called the Gioux correspondence. We will recall certain pieces of this celebrated statement in detail. 
\begin{thm}\label{thm10}(Theorem 10 in \cite{openbook}) Any contact structure on a closed manifold is supported by an open book with Weinstein pages.   
\end{thm}
The above statement is the core part of the correspondence between supported open books and contact structures. In fact the existence statement for the opposite direction is relatively easy to achieve, especially in dimension three. Namely, given an open book in a 3-dimensional closed manifold, it is not hard to construct a contact form on the corresponding abstract open book, whose kernel is supported. It turns out that in higher dimensions, one needs to a have an exact symplectic page and a symplectic monodromy in order to construct a contact form on an abstract open book, whose kernel is supported, see Proposition 9 in \cite{openbook} and Proposition 17 in \cite{ILD}. We will carry out such a construction in Proposition \ref{niceform}. Concerning the uniqueness features of the Giroux correspondence, we are mainly interested in one side, namely the "uniqueness" of supported contact structures. It turns out that such a statement is again more involved in higher dimensions. Philosophically, given an open book, the symplectic geometry of the pages determines the supported contact structures and in dimension three, any two symplectic structure on a page are isotopic since they are simply two area forms on a given surface. But in higher dimensions, this is not true in general.  

In \cite{ILD}, Giroux introduced the notion of a Liouville open book, which clears out the technicalities that pointed above. 

A \emph{Liouville open book}, abbreviated as  LOB, in a closed manifold $V$ is a tripple $(K,\Theta,(\omega_x)_{x\in S^1})$ where
\begin{enumerate}[label=\text{(lob\arabic*)}]
\item \label{lob1}$(K,\Theta)$ is an open book on $V$ with pages $F_x=\Theta^{-1}(x)\cup K$, $x\in S^1$;
\item \label{lob2} $(F_x,\omega_x)$ is an ILD for all $x\in S^1$ and the following holds:  there is a defining function $h:V\rightarrow \mathbb{C}$ for $(K,\Theta)$ and a $1$-form  $\beta$ on $V$ such that the restriction of $d(\beta/|h|)$ to each page is an ILF. More precisely, 
$$\omega_x=d(\beta/|h|)_{|TF_x^o}$$
for all $x\in S^1$.
\end{enumerate}   
The 1-form $\beta$ in \ref{lob2} is called  a \emph{binding 1-form} associated to $h$. Note that if $h'$ is another defining function for $(K,\Theta)$, then  $h'=\kappa h$  for some positive function $\kappa$ on $V$ and $\beta':=\kappa \beta$ is a binding 1-form associated to $h'$. We also note that for a fixed defining function, the set of associated binding 1-forms is an affine space. 

Similar to classical open books, LOB's are characterized by the monodromy, which now has to be symplectic. Namely, one  considers a \emph{symplectically spinning vector field}, that is a vector filed $X$ satisfying \ref{m1}-\ref{m2} and generating the kernel of a closed 2-form on $V\setminus K$, which restricts to $\omega_x$ for all $x\in S^1$. Given such a vector field, the time-one-map of its flow, say $\phi$, is a diffeomorphism of $F:=F_0$, which fixes $K$ and preserves $\omega:=\omega_0$. The isotopy class $[\phi]$, among the symplectic diffeomorphisms that fixes $K$, is called the \emph{symplectic monodromy} and characterizes the given LOB. For the construction of a LOB in the abstract open book $OB(F,\phi)$, where $\phi^*\omega=\omega$, we refer to Propostion 17 in \cite{ILD} and Proposition \ref{niceform}. 

Similar to the classical open books, symplectically spinning vector fields form an affine space and all representatives of the symplectic monodromy may be obtained by sweeping out this affine space. It turns out that the obvious choice of a symplectically spinning vector field is actually smooth and by modifying a given binding 1-form along $\Theta$, it is possible to get a symplectically spinning vector filed, whose flow is 1-periodic near the binding. 
\begin{lem}\label{lem15}(Lemma 15 in \cite{ILD}) Let $(K, \Theta,(\omega_x)_{x\in S^1})$ be a LOB on a closed manifold $V$ and $h: V\rightarrow \mathbb{C}$ be a defining function for $(K,\Theta)$. Then for every
binding 1-form $\beta$, the vector field $X$ on $V\setminus K$ spanning the kernel of $d(\beta/|h|)$ and
satisfying $(\Theta^*dx)(X)=2\pi$ extends to a smooth vector field on $V$ which is zero along K.
Furthermore, $\beta$ can be chosen so that $X$ is 1-periodic near K.
\end{lem}
Natural sources of LOBs are contact manifolds, namely we have the following statement. 
\begin{prop}\label{prop18} (Proposition 18 in \cite{ILD}) Let $(V, \xi)$ be a closed contact
manifold, and $(K, \Theta)$ be a supporting open book with defining function $h: V\rightarrow \mathbb{C}$.
Then the contact forms $\alpha$ on $(V,\xi)$ such that $d(\alpha/|h|)$ induces an ideal Liouville structure
on each page form a non-empty convex cone.
\end{prop}

Let $(K,\Theta,(\omega_x)_{x\in S^1})$ be a LOB on a closed manifold $V$ with a defining function $h$. A co-oriented contact structure $\xi$ on $V$ is said to be  \emph{symplectically supported} by $(K,\Theta,(\omega_x)_{x\in S^1})$ if there exists a contact form $\alpha$ on $(V,\xi)$ such that $\alpha$ is a binding 1-form of the LOB associated to $h$. 

By our remark following the definition of the binding 1-form, the definition of being symplectically supported is independent of the given defining function. But the crucial fact is that once a defining function is fixed, a contact binding 1-form is unique whenever it exists, see Remark 20 in \cite{ILD}. Hence, once a defining function $h$ is fixed, there is a one-to-one correspondence between contact structures supported by $(K,\Theta,(\omega_x)_{x\in S^1})$ and contact binding 1-forms associated to $h$. Now given two contact structures $\xi_0$ and $\xi_1$ supported by $(K,\Theta,(\omega_x)_{x\in S^1})$, there exist unique contact binding 1-forms $\alpha_0$ and $\alpha_1$ respectively. Since the set of binding 1-forms associated to $h$ is affine, there is a path $(\beta_t)_{t\in [0,1]}$ of binding 1-forms such that $\beta_0=\alpha_0$ and $\beta_1=\alpha_1$. Then by modifying $\beta_t$'s along the 1-form $\Theta^*dx$, one gets  a path of contact forms $(\beta^c_t)_{t\in [0,1]}$ and a homotopy $\left((\beta^s_t)_{t\in [0,1]}\right)_{s\in [0,c]}$ between the paths $(\beta_t)_{t\in[0,1]}$ and $(\beta^c_t)_{t\in [0,1]}$ such that 
\begin{itemize}
\item for all $s\in [0,c]$ and $t\in[0,1]$, $\beta^s_t$ is a binding 1-form for $(K,\Theta,(\omega_x)_{x\in S^1})$ associated to $h$ (since $\beta_t$'s stay the same along the pages through the modification);
\item for all $s\in [0,c]$, $\beta^s_0$ and $\beta_1^s$ are contact forms (since if $\beta_t$ is already a contact form then it keeps being a contact form through the modification).
\end{itemize}  
In particular, whenever $\beta^s_t$ is a contact form, $\ker \beta^s_t$ is symplectically supported by $\left(K,\Theta,(\omega_x)_{x\in S^1}\right)$ and $\beta_t^s$ is the unique contact binding 1-form associated to $h$. This tells us that the concatenation of the paths 
$(\ker \beta_0^s)_{s\in[0,c]}$, $(\ker \beta^c_t)_{t\in [0,1]}$ and $(\ker \beta^{c-s}_1)_{s\in [0,c]}$ gives an isotopy between $\xi_0$ and $\xi_1$ along the contact structures that are symplectically supported by  $(K,\Theta,(\omega_x)_{x\in S^1})$. In fact the following more general statement holds. 
\begin{prop}\label{prop21}(Proposition 21 in \cite{ILD}) On
a closed manifold, contact structures supported by a given Liouville open book
form a non-empty and weakly contractible subset in the space of all contact structures.
\end{prop}

Now we are ready to construct the special contact form on given contact manifold $(V,\xi)$. Our construction is the generalization of Proposition 1 in \cite{general} to any dimension.
\begin{prop}\label{niceform}
Let $(V,\xi)$ be a closed connected contact manifold. Then there is an embedded compact hypersurface $F\subset V$
with the following property. Given any $\epsilon >0$, there exists a contact form $\alpha$ on $(V,\xi)$ such that $F$ is a global hypersurface of section for the Reeb flow of $\alpha$ and the followings hold.
\begin{enumerate}[label=\text{(F\arabic*)}]
\item \label{F1} $\alpha$ restricts to a contact form on $K:=\partial F$, for which the closed Reeb orbits have period at least $1/2$.
\item \label{F2} The first return time function 
$$\tau :F\setminus \partial F \rightarrow (0,+\infty)$$
of the Reeb flow of $\alpha$ extends to a smooth function on $F$ and 
the corresponding first return map 
$$\varphi :F\setminus \partial F \rightarrow F\setminus \partial F$$
extends to a diffeomorphism of $F$.   
\item \label{F3} There exists an open subset $U\subset F$ such that the support 
of $\varphi$ is contained in $F\setminus U$ and $\tau \equiv 1$ on $U$. 
\item \label{F4}  $$||\tau-1||_\infty <\min \{\epsilon,\frac{\epsilon}{{\rm vol}(K, \alpha)}\}.$$ 
\item \label{F5} $$\int _{F\setminus U}(d\alpha)^n<\epsilon$$
where $2n+1=\dim V$. 
\end{enumerate}
\end{prop}
\begin{proof}
We prove the statement by induction on $\dim V=2n+1$. For $n=1$ the statement follows from the Propositon 1 in \cite{general}. In fact, \ref{F1} follows from the property (i) and  \ref{F4} is a harmless modification of the property (iv). Now assume that 
the statement is true for $n-1$. 

Let $(V,\xi)$ be given such that $\dim V=2n+1$. By Theorem \ref{thm10}, there is an open book $(K,\Theta)$ in $V$ supported by $\xi$. Let $F_x:=\Theta^{-1}(x)$, $x\in S^1=\mathbb{R}/2\pi\mathbb{Z}$ denote the pages of the open book and let $h:V\rightarrow \mathbb{C}$ be a defining function for $(K,\Theta)$.
We want to show that the page 
\begin{eqnarray}\label{0thpage}
F:=\Theta^{-1}(0)\cup K
\end{eqnarray}
is a hypersurface that satisfies (F1)-(F5). 

Let $\epsilon>0$ be given. By Proposition \ref{prop18}, there is a contact form $\alpha$ 
on $(V,\xi)$ such that $(K,\Theta, d(\alpha/|h|)_{TF_x^o})$ is a LOB, which supports $\xi$ symplectically. By Lemma \ref{lem15}, we modify the binding 1-form $\alpha$ and obtain a binding 1-form  $\beta$ such that the associated symplectically spinning vector field $X$ is 1-periodic near $K$. 
Hence the time-one-map of the flow of $X$ gives us  a diffeomorphism 
$\phi:F\rightarrow F$ such that  
\begin{equation}\label{monodromy}
\phi^*(d\lambda)=d\lambda
\end{equation}
where $\lambda\in \Omega^1(F^o)$ is the ILF given by 
\begin{eqnarray}\label{lambda}
\lambda:=(\beta/|h|)_{|TF^o}=(\alpha/|h|)_{|TF^o}
\end{eqnarray}
and $\phi={\rm id}$ on some neighbourhood of $K$ in $F$. Now our aim is to recover $V$ as the 
abstract open book induced by the pair $(F,\phi)$ and to define a contact form on the abstract open book with the desired properties. 

In order to construct the desired contact form, we first need a suitable contact form on $K=\partial F$. We note that 
$(K, \ker (\alpha_{|TK}))$ is a $(2n-1)$-dimensional closed connected contact manifold. In fact, since $F$ is a Weinstein domain, see Theorem \ref{thm10}, and $\dim F\geq 4$, $\partial F$ is connected. By the inductive hypothesis, there is a compact hypersurface $F'\subset K$ such that for a given $\epsilon'>0$, there is contact form $\alpha'$ on $(K,\ker (\alpha_{|TK}))$ with the following properties.
\begin{enumerate}[label=\text{(F'\arabic*)}]
\item \label{F'1} $\alpha'$ restricts to a contact form on $K':=\partial F'$, for which the closed Reeb orbits have period at least $1/2$.
\item \label{F'2} The first return time function 
$$\tau' :F'\setminus \partial F' \rightarrow (0,+\infty)$$
of the Reeb flow of $\alpha'$ extends to a smooth function on $F'$ and 
the corresponding first return map 
$$\varphi' :F'\setminus \partial F' \rightarrow F'\setminus \partial F'$$
extends to a diffeomorphism of $F'$.   
\item \label{F'3} There exists an open subset $U'\subset F'$ such that the support 
of $\varphi'$ is contained in $F'\setminus U'$ and $\tau' \equiv 1$ on $U$. 
\item \label{F'4}  $$||\tau'-1||_\infty <\min \{\epsilon',\frac{\epsilon'}{{\rm vol}(K', \alpha')}\}.$$ 
\item \label{F'5} $$\int _{F'\setminus U'}(d\alpha')^{n-1}<\epsilon'.$$ 
\end{enumerate}
For later convenience, we define the contact form 
\begin{eqnarray}\label{alpha_0}
\alpha_0:=2\pi \alpha'
\end{eqnarray}
on $K$. 

Second we need a nice parametrization of $F$ near $K$. By Lemma \ref{nearKlemma}, we have an embedding 
\begin{eqnarray}\label{lambdanearK}
\imath:[0,R]\times K \rightarrow F;\; \imath^*\lambda=\alpha_0/r,\;r\neq 0
\end{eqnarray}
and $\imath(0,q)=q$ for all $q\in K$.  

Following the line of arguments in \cite{general}, we want to construct the desired contact form on the open book defined by the pair $(F,\phi)$ using a Liouville form on $F^o$, which has a particular behaviour near K. We cook up such a Liouville form out of $\lambda$ as follows. 
Since $\phi={\rm id}$ near $K$, there is some $\rho'>0$ such that $\rho'\leq R$ and  
\begin{eqnarray}\label{phiisid}
\phi={\rm id} \textrm{ on }[0,\rho']\times K.
\end{eqnarray}
Then there exist numbers $C>0$, $\rho<\rho'$ and a smooth function 
$$v:[0,\rho']\rightarrow (0,+\infty)$$
such that the followings hold
\begin{enumerate}[label=(v\arabic*)]
\item \label{defnofv1} $v(r)=\frac{1}{r}$ near  $\rho'$
\item \label{defnofv2} $v(r)=C(1-r^2)$  for $r\leq \rho$
\item \label{defnofv3} $v'(r)<0$  for $r\in(0,\rho']$. 
\end{enumerate}
Finally we define the 1-form $\eta$ on $F$ via 
\begin{eqnarray}\label{eta}
\eta= \left\{
\begin{array}{ll}
       \lambda/C &\textrm{on }\; ([0,\rho']\times K)^c \\
       v(r)\alpha_0/C &\textrm{ on }\; [0,\rho']\times K\\
\end{array} 
\right. 
\end{eqnarray}
By \ref{defnofv1}, $\eta$ is a smooth 1-form and we note that on $(0,\rho')\times K$, 
$$Cd\eta=v'dr\wedge \alpha_0+vd\alpha_0\;\Rightarrow C^n(d\eta)^n=(n-1)v'v^{n-1}dr\wedge\alpha_0\wedge (d\alpha_0)^{n-1}.$$
Hence by \ref{defnofv3}, $d\eta$ is symplectic on $(0,\rho')\times K$ and therefore on $F\setminus K$.

Now we are ready to construct the open book associated to the pair $(F,\phi)$ and the required contact form. We first consider the mapping torus 
$$MT(F,\phi):=([0,2\pi]\times F )\big/\sim \,;\; (2\pi,q)\sim (0,\phi(q)).$$
Via (\ref{lambdanearK}) and (\ref{phiisid}), we have a neighborhood
$W$ of $\partial MT(F,\phi)$ in $MT(F,\phi)$ with the coordinates 
\begin{eqnarray}\label{coordinatesinW}
(x,q,r)\in W:=\mathbb{R}/2\pi\mathbb{Z}\times K\times [0,\rho]
\end{eqnarray}
 and the open book is given by
\begin{eqnarray}\label{OB}
OB(F,\phi):=\big(MT(F,\phi)\sqcup K\times \rho\mathbb{D}\big)\big/\sim\,; (x,q,r)\sim (q,re^{ix}).
\end{eqnarray}
Then the set 
$$W':=K\times \rho \mathbb{D}$$
is a compact neighbourhood of $K=\{0\}\times K$ in $OB(F,\phi)$, containing $W\setminus \partial MT(F,\phi)$.

On $[0,2\pi]\times F$, we define a 
family of 1-forms 
\begin{eqnarray}\label{alphatilda_s}
\tilde{\alpha}_s=dx+s\big((1-\beta(x))\eta+\beta(x)\phi^*\eta\big)
\end{eqnarray}
where $s$ is a positive real parameter and $\beta:[0,2\pi]\rightarrow [0,1]$ is 
a smooth function such that $\beta(0)=0$, $\beta(2\pi)=1$ and ${\rm supp}
(\beta')\subset (0,2\pi)$. By the 
choice of $\beta$, $\tilde{\alpha}_s$ descends to a family of 1-forms on 
$MT(F,\phi)\setminus \partial MT(F,\phi)$ and equivalently on $OB(F,\phi)\setminus K$.
We note that on $W$, (\ref{alphatilda_s}) reads as
\begin{eqnarray}\label{alphatilda_sonW}
\tilde{\alpha}_s=dx+s(1-r^2)\alpha_0.
\end{eqnarray}
We fix $\delta>0$ and as in the proof of Proposition 1 in \cite{general}, there exists $s_1>0$ depending on $\rho$ and $\delta$ so that the following holds. 
For any $s\in(0,s_1)$, there is a curve 
$$\gamma:[0,\rho]\rightarrow \mathbb{C},\; \gamma(r)=f(r)+ig(r)$$
with $f,g\geq 0$ satisfying
\begin{enumerate}[label=\text{(g\arabic*)}]
\item\label{A1} $\gamma(r)=1+is(1-r^2)$ on $[r_1,\rho]$, for some $r_1\in (0,\rho)$.
\item \label{A2} $g'<0$ on $(0,\rho]$.
\item \label{A3}  There exits some $r_0\in (0,r_1)$ such that 
$$\gamma (r)=r^2+i(1+\delta-r^2)$$
for $r\in[0,r_0]$ and 
$$g(r_0)-g(\rho)=g(r_0)-s(1-\rho^2)\leq 2\delta.$$
\item\label{A4} $\frac{g'f-gf'}{g^2+f^2}<0$ on $(0,\rho]$.
\item\label{A5} $\frac{g''f'-f''g'}{(g')^2+(f')^2}\leq 0$ on $[0,\rho]$.
\end{enumerate}
Now we define a family of smooth 1-forms on $OB(F,\phi)$ by
\begin{eqnarray}\label{alpha_sgeneral}
\alpha_s=   \left\{
\begin{array}{ll}
      \frac{\tilde{\alpha}_s}{2\pi (1+\delta)} &\textrm{ on } OB(F,\phi)\setminus W'\\
      \\
      \frac{f(r)dx+g(r)\alpha_0}{2\pi (1+\delta)} &\textrm{ on } W'\\
\end{array} 
\right.
\end{eqnarray}
By \ref{A3} one has 
$$\alpha_s=\frac{r^2dx+(1+\delta-r^2)\alpha_0}{2\pi (1+\delta)},$$
near $K\subset W'$ 
so that $\alpha_s$ is a smooth 1-form on $W'$ for all $s\in (0,s_1)$ and by (\ref{alpha_0}),
\begin{eqnarray}\label{alpha_sonK}
 \alpha_s=\frac{1}{2\pi}\alpha_0=\alpha'\; \textrm{ on}\; K.
\end{eqnarray}
We note that for $\epsilon'<1/2$, any closed Reeb orbit of $\alpha'$ that passes through $F'\setminus K'$ has period at least $1/2$ due to \ref{F'4}. Moreover, by \ref{F'1}, any closed Reeb orbit of $\alpha'$ that is contained in the contact submanifold $K'\subset K$ has period at least $1/2$. Hence for all $s\in (0,s_1)$ the condition \ref{F1} is satisfied for $\alpha_s$ if $\epsilon'<1/2$. 

\begin{lem}\label{lemcontact} There exists $s_2\in (0,s_1)$, depending on $\delta, \rho, \phi,\eta, \beta$ such that 
$\alpha_s$ is a contact form on $OB(F,\phi)$ for all $s\in (0,s_2)$.
\end{lem}
\begin{proof} We first check the statement on $W'$. We compute
\begin{eqnarray*}
\alpha_s\wedge (d\alpha_s)^{n}
&=&\frac{1}{(2\pi(1+\delta))^{n+1}}\big[(f dx+g\alpha_0)\wedge(f'dr\wedge dx+ g'dr\wedge \alpha_0 + gd\alpha_0)^n \big]\\
&=&\frac{(n-1)g^{n-1}}{(2\pi(1+\delta))^{n+1}}\big[(f dx+g\alpha_0)\wedge\big(f'dr\wedge dx\wedge(d\alpha_0)^{n-1}\\
&& +g'dr\wedge d\alpha_0\wedge(d\alpha_0)^{n-1}\big)\big]\\
&=&\frac{(n-1)g^{n-1}(fg'-gf')}{(2\pi(1+\delta))^{n+1}}\big( dx\wedge dr\wedge\alpha_0 \wedge (d\alpha_0)^{n-1}\big)\\
&=&\frac{(n-1)g^{n-1}(f'g-fg')}{r(2\pi(1+\delta))^{n+1}}\big( (r dr\wedge dx)\wedge \alpha_0 \wedge (d\alpha_0)^{n-1}\big)
\end{eqnarray*}
Now for $r\neq 0$, $g>0$ and by \ref{A4} we get
$$\frac{(n-1)g^{n-1}(f'g-fg')}{r(2\pi(1+\delta))^{n+1}}>0.$$
For $r$ close to 0, by \ref{A3}, we get 
\begin{eqnarray*}
\frac{(n-1)g^{n-1}(f'g-fg')}{r(2\pi(1+\delta))^{n+1}}
&=&\frac{(n-1)g^{n-1}(2r(1+\delta-r^2)-r^2(-2r))}{r(2\pi(1+\delta))^{n+1}}\\
&=&\frac{(n-1)g^{n-1}2r(1+\delta)}{r(2\pi(1+\delta))^{n+1}}\\
&=&\frac{2(n-1)(1+\delta-r^2)^{n-1}}{(2\pi)^{n+1}(1+\delta)^{n}}.
\end{eqnarray*}
We note that the limit of this expression is positive as $r$ tends to 0. Hence $\alpha_s$ is a positive contact form for all $s\in (0,s_1)$.

Next we consider $OB(F,\phi)\setminus W'$. We have
\begin{eqnarray*}
2\pi (1+\delta)\,d\alpha_s
&=&s\big(-\beta'dx\wedge\eta+(1-\beta)d\eta+ \beta'dx\wedge \phi^*\eta+\beta \phi^*d\eta\big)\\
&=&s\big(\beta'dx\wedge(h^*\eta-\eta)+(1-\beta)d\eta +\beta d\eta\big)\\
&=&s \big(\beta'dx\wedge \eta_\Delta+d\eta\big)
\end{eqnarray*}
where $\eta_\Delta:=\phi^*\eta-\eta$. The second equation follows from the fact that $\phi$ is a symplectomorphism. We get 
$$(2\pi (1+\delta))^n\,(d\alpha_s)^n=s^n\big((n-1)\beta'dx\wedge \eta_\Delta\wedge (d\eta)^{n-1}+(d\eta)^n\big).$$
Hence up to a positive constant we have
\begin{eqnarray*}
\frac{\alpha_s\wedge (d\alpha_s)^{n}}{s^n}
&=&\big[dx+s\big(\eta+\beta\eta_\Delta)\big]\wedge\big[(n-1)\beta'dx\wedge \eta_\Delta\wedge (d\eta)^{n-1}+(d\eta)^n\big]\\
&=&dx\wedge (d\eta)^n+s(n-1)\beta'\eta \wedge dx\wedge \eta_\Delta\wedge (d\eta)^{n-1}.
\end{eqnarray*}
Now since $dx\wedge (d\eta)^n$ is a volume form, we can choose $s_2\in (0,s_1)$ small enough so that 
$\alpha_s$ is a contact form on $OB(F\phi)\setminus W'$ for all $s\in (0,s_2)$. 
\end{proof}

After identifying $V$ with the abstract open book $OB(F,\phi)$, for each $s\in (0,s_2)$, we have two contact forms on $OB(F,\phi)$, namely the one induced by $\alpha$ on $(V,\xi)$, still denoted by $\alpha$ and $\alpha_s$ defined above. We note that the statement of the proposition is invariant under a diffeomorphism, hence it is enough to show that $F\subset OB(F,\phi)$ has the following property: given $\epsilon>0$, there is a contact form $\alpha_s$ such that the statements \ref{F1}-\ref{F5} hold for $\alpha_s$ and $\ker \alpha_s$ isotopic to $\ker \alpha$.

In order to relate $\ker\alpha_s$ to $\ker \alpha$,  we first want to show that the obvious open book structure on $OB(F,\phi)$ is a Liouville open book with the binding form $\alpha_s$. Let 
$$\hat{\Theta}: OB(F,\phi)\setminus K \rightarrow S^1$$
be the fibration induced by the projection $MT(F,\phi)\rightarrow S^1$. We pick a suitable defining function $\hat{h}$ as follows. We define a smooth function
$$\hat{u}:F\rightarrow [0,\infty)$$
such that for some suitably chosen $b>0$,
\begin{enumerate}[label=\textrm{(\^{u}\arabic*)}]
\item \label{B1} $\hat{u}(r,q)=r$ for $(r,q)\in [0,\rho]\times K$,
\item \label{B2} $\hat{u}\equiv b$ on $([0,\rho')\times K)^c$,
\item \label{B3} $\hat{u}$ depends only on $r$ and $\hat{u}_r\geq 0$ on $[0,\rho']\times K$.
\end{enumerate}

We note that on ${\rm supp}(\phi)$, $\hat{u}$ is constant. Hence the $S^1$-equivariant extension of $\hat{u}$ is a well-defined smooth function on $MT(F,\phi)$, which constitutes the function $|\hat{h}|$, and as it is a linear function of $r$ near $K$, pairing $|\hat{h}|$ with $\hat{\Theta}$ leads to a well-defined defining function $\hat{h}$ for the open book $(K,\hat{\Theta})$ on $OB(F,\phi)$. 
\begin{lem} \label{alpha_sbinds} For any $s\in(0,s_2)$ , $d(\alpha_s/|\hat{h}|)$ induces an ideal Liouville structure on each fibre of $\hat{\Theta}$.
\end{lem}
\begin{proof} We put
\begin{eqnarray}\label{lambda_s}
\lambda_x^s:=(\alpha_s/|\hat{h}|)_{|TF^o_x}
\end{eqnarray} 
where $F_x=\hat{\Theta}^{-1}(x)$. We study $\lambda_x^s$ on pieces of $F_x$ separately. 
\begin{itemize}
\item \underline{On $\{x\}\times (0,\rho]\times K$}: by \ref{B1}, we have 
\begin{eqnarray}\label{lambda_sbeforerho}
\lambda^s_x=\frac{g(r)\alpha_0}{2\pi(1+\delta)r}.
\end{eqnarray}
Hence up to  positive constants, we get 
$$d\lambda_x^s=\frac{g'r-g}{r^2}\,dr\wedge \alpha_0 +\frac{g}{r}\,d\alpha_0$$
$$\Rightarrow\; (d\lambda^s_x)^n=(n-1)g^{n-1}\frac{g'r-g}{r^{n+1}}\,dr\wedge \alpha_0\wedge(d\alpha_0)^{n-1}.$$
We note that due to the parametrization (\ref{lambdanearK}), 
$dr\wedge \alpha_0\wedge(d\alpha_0)^{n-1}$ is a negative volume form.  By \ref{A2}, $g'<0$ and $g'r-g<0$ since $g\geq 0$. Hence $d\lambda_x^s$ is a positive symplectic form for all $s$.
\newline
\item \underline{On $\{x\}\times (\rho,\rho')\times K$}: we note that $\phi={\rm id}$ on this set. Hence, up to a positive constant, we have
\begin{eqnarray}\label{lambda_sbetweenrhoandrho'}
\lambda^s_x=s\frac{\eta}{\hat{u}}
\end{eqnarray}
By \ref{B3}, we have
$$s^{-1}d\lambda^s_x=-\frac{1}{\hat{u}^2}\hat{u}_r\,dr\wedge\eta+\frac{1}{\hat{u}}\,d\eta$$
$$\Rightarrow\; s^{-n}(d\lambda^s_x)^n=-\frac{(n-1)}{\hat{u}^{n+1}}\hat{u}_r\,dr\wedge\eta\wedge (d\eta)^{n-1}+\frac{1}{\hat{u}^n}\,(d\eta)^n.$$
We also note that
$$\eta=\frac{v}{C}\alpha_0\;\Rightarrow\; C(d\eta)=v'\,dr\wedge\alpha_0+v\,d\alpha_0$$
$$\Rightarrow\;C^{n-1}(d\eta)^{n-1}=(n-2)v^{n-2}v'\, dr\wedge\alpha_0\wedge(d\alpha_0)^{n-2}+v^{n-1}\,(d\alpha_0)^{n-1},$$
$$\Rightarrow\;C^{n}(d\eta)^{n}=(n-1)v^{n-1}v'\, dr\wedge\alpha_0\wedge(d\alpha_0)^{n-1},$$
$$\Rightarrow\;C^{n}\,\eta\wedge(d\eta)^{n-1}=v^{n}\,\alpha_0\wedge (d\alpha_0)^{n-1}.$$
Combining all these computations, we get
$$s^{-n}C^n\,(d\lambda^s_x)^n=-\frac{(n-1)}{\hat{u}^{n+1}}\hat{u}_rv^n\,dr\wedge\alpha_0\wedge (d\alpha_0)^{n-1}+\frac{1}{\hat{u}^n}(n-1)v^{n-1}v'\, dr\wedge\alpha_0\wedge(d\alpha_0)^{n-1}$$
$$\Rightarrow\;s^{-n}C^n\,(d\lambda^s_x)^n=\bigg[-\frac{(n-1)}{\hat{u}^{n+1}}\hat{u}_rv^n+\frac{1}{\hat{u}^n}(n-1)v^{n-1}v'\bigg]
\, dr\wedge\alpha_0\wedge(d\alpha_0)^{n-1}.$$
By \ref{defnofv3}, $v'<0$ and by \ref{B2}, $\hat{u}\geq 0$. Hence $\lambda^s_x$ is a 
positive symplectic form.
\newline
\item \underline{On $\{x\}\times ((0,\rho')\times K)^c$}: we have $\hat{u}\equiv b$ by \ref{B2}. Hence 
\begin{eqnarray}\label{lambda_safterrho'}
\lambda^s_x=s\big[(1-\beta(x))\frac{\eta}{b}+\beta(x)\frac{\phi^*\eta}{b}\big]\;\Rightarrow\;d\lambda^s_x=s\frac{d\eta}{b}.
\end{eqnarray}
Since $d\lambda^s_x$ coincides with $d\lambda$ up to a positive constant, it is symplectic for all $s$.
\end{itemize} 
\end{proof}
Now we are in the following situation. On $OB(F,\phi)$, we have the Liouville open book
\begin{eqnarray}\label{LOB1}
(K,\hat{\Theta},d(\alpha/|h|)_{|TF^0_x}),
\end{eqnarray}
which is symplectically supported by the contact structure $\ker \alpha$ and for any $s\in (0,s_2)$ we have the second Liouville open book
\begin{eqnarray}\label{LOB2}
(K,\hat{\Theta},d(\alpha_s/|\hat{h}|)_{|TF^0_x}),
\end{eqnarray}
which is symplectically supported by the contact structure $\ker \alpha_s$. We want to show that 
there exists a diffeomorphism 
\begin{eqnarray}\label{Phi}
\Phi: OB(F,\phi)\rightarrow OB(F,\phi)
\end{eqnarray}
such that $\Phi\circ \hat{\Theta}=\hat{\Theta}\circ \Phi$ and the restriction of $\Phi$ to each fibre is symplectic, that is, for all $x\in S^1$,
$$\Phi^*d(\alpha_s/|\hat{h}|)_{|TF^0_x}=d(\alpha/|h|)_{|TF^0_x}.$$
If such a diffeomorphism exists, then  $\ker (\Phi_*\alpha_s)$ and $\ker \alpha$ are two contact structures on $OB(F,\phi)$, which symplectically support the Liouville open book (\ref{LOB1}). Hence they are isotopic  by Proposition \ref{prop21}. Then it is enough to establish the dynamical properties given by \ref{F1}-\ref{F5}  for $F$ and $\alpha_s$. Now our task is to construct the diffeomorphism $\Phi$. 
\begin{lem} Let $s\in (0,s_2)$ be fixed. Then there exists a diffeomorphism 
\begin{eqnarray}
\Phi: OB(F,\phi)\rightarrow OB(F,\phi)
\end{eqnarray}
such that $\Phi\circ \hat{\Theta}=\hat{\Theta}\circ \Phi$ and the restriction of $\Phi$ to each fibre is symplectic, that is, for all $x\in S^1$,
$$\Phi^*d(\alpha_s/|\hat{h}|)_{|TF^0_x}=d(\alpha/|h|)_{|TF^0_x}.$$
\end{lem}
\begin{proof}
We first take a closer look at the setting. By the equations (\ref{lambda_sbeforerho}), (\ref{lambda_sbetweenrhoandrho'}) and (\ref{lambda_safterrho'}), the ideal Liouville structures $(d(\alpha_s/|\hat{h}|)_{|TF^0_x})_{x\in S^1}$ may be identified with the ideal Liouville structure on the 0-th page, namely
\begin{equation}\label{omega^s}
\omega^s:=d(\alpha_s/|\hat{h}|)_{|TF^0}.
\end{equation}
Similarly, the symplectic form 
\begin{equation}\label{omega}
\omega:=d\lambda=d(\alpha/|h|)_{|TF^0}
\end{equation}
gives the ideal Liouville structre on each page $\{x\}\times F\cong F$,
after identifying $V$ with $OB(F,\phi)$ via the flow of the symplectically spinning vector field $X$.  

We first show that 
$$\omega_t:=(1-t)\omega+t\omega^s$$
is symplectic on $F^o$ for all $t\in [0,1]$. In fact, we claim that
\begin{eqnarray}\label{lambda_t}
\lambda_t=(1-t)\lambda+t\lambda^s,\;t\in [0,1]
\end{eqnarray}
is a Liouville form on $F^o$ for all $t$, where $\lambda$ is the primitive of $\omega$ given by (\ref{lambda}) and  $\lambda^s$ is the primitive of $\omega^s$ given by (\ref{lambda_s}).
Again we compute $d\lambda_t$ on separate pieces of $F^o$. 
\begin{itemize}
\item \underline{On $(0,\rho]\times K$}: we have 
$$\lambda_t=(1-t)\frac{\alpha_0}{r}+t\frac{g(r)\alpha_0}{2\pi(1+\delta)\hat{u}(r)}=\big[(1-t)\frac{1}{r}+t\frac{g}{2\pi(1+\delta)\hat{u}}\big]\alpha_0=:\kappa(r)\alpha_0$$
$$\Rightarrow\; d(\lambda_t)=\kappa'\,dr\wedge \alpha_0+\kappa\,d\alpha_0$$
$$\Rightarrow\; (d\lambda_t)^n=(n-1)\kappa^{n-1}\kappa'\,dr\wedge \alpha_0\wedge (d\alpha_0)^{n-1}.$$
Due to the parametrization (\ref{lambdanearK}), $dr\wedge \alpha_0\wedge (d\alpha_0)^{n-1}$ is a negative volume form. Hence $d\lambda_t$ is a positive symplectic form if and only if $\kappa'<0$.
By \ref{B1} we have
$$\kappa=(1-t)\frac{1}{r}+t\frac{g(r)}{2\pi (1+\delta)r}$$
$$\Rightarrow\; \kappa'=-(1-t)\frac{1}{r^2}+\frac{t}{2\pi (1+\delta)r^2}(rg'-g)>0$$
since $g'<0$ on $(0,\rho]$.
\newline
\item \underline{On $(\rho,\rho']\times K$}: (\ref{lambda_t}) reads as
$$\lambda_t=(1-t)\frac{1}{r}\alpha_0+t\frac{sv}{C2\pi(1+\delta)\hat{u}}\alpha_0=\big[(1-t)\frac{1}{r}+t\frac{sv}{C2\pi(1+\delta)\hat{u}}\big]\alpha_0=:\kappa(r)\alpha_0$$
We have
$$\kappa'=-(1-t)\frac{1}{r^2}+\frac{st}{C2\pi (1+\delta)\hat{u}^2}(\hat{u}v'-v\hat{u}').$$
By \ref{defnofv3} and \ref{B3}, $\hat{u}v'-v\hat{u}'>0$ and therefore $\kappa'<0$.
\newline
\item \underline{On $((0,\rho']\times K)^c$}: by (\ref{lambda_safterrho'}), we have
$$\lambda_t=(1-t)\lambda+ ts\frac{\eta}{b}=\big[(1-t)+\frac{ts}{bC}\big]\lambda.$$
Since $d\lambda$ is positive symplectic and  
$$(1-t)+\frac{ts}{bC}>0,$$
$d\lambda_t$ is a positive symplectic form for all $s$ and $t$.
\end{itemize}
Hence $\omega_t=d\lambda_t$ is symplectic on $F^o$ for all $t\in [0,1]$.

We note that for each $t\in[0,1]$, $u\lambda_t$ extends to $F$, the extension being the smooth $1$-form 
$$(1-t)\alpha_{|TF}+t\frac{u}{\hat{u}}(\alpha_s)_{|TF},$$
where the function $u$ on $F$ is given by $|h|_{|F}$ and $h$ is the initial defining function for $(K,\Theta)$. After reparametrizing $[0,1]$ as $[0,2\pi]$, we get a smooth path of ideal Liouville structures $(\omega_x)_{x\in [0,2\pi]}$ in the sense of Lemma \ref{lem6} so that $\omega_0=\omega$ and $\omega_{2\pi}=\omega_s$.
It is also clear that for each $x\in [0,2\pi]$, the boundary contact structure associated to $\omega_x$ is $\ker \alpha_0$. Hence by Lemma \ref{lem6}, there exists a smooth isotopy $(\psi_x)_{x\in [0,2\pi]}$ of $F$ such that 
\begin{enumerate}[label=\text{(i\arabic*)}]
\item\label{psi1} $\psi_0={\rm id}$;
\item \label{psi2} $\psi_x={\rm id}$ on $K$ for all $x\in [0,2\pi
]$;
\item \label{psi3} $\psi_x^*\omega_x=\omega_0=\omega$ for all $x\in [0,2\pi]$.
\end{enumerate}
Now we define $\Phi:[0,2\pi]\times F\rightarrow [0,2\pi]\times F$
\begin{eqnarray}\label{def_Psi}
\Phi(x,p):=(x,\psi_{2\pi}\circ\psi^{-1}_x\circ \phi^{-1}\circ \psi_x(p)).
\end{eqnarray}
We note that 
$$\Phi(2\pi,p)=(2\pi,\psi_{2\pi}\circ \psi^{-1}_{2\pi} \phi^{-1}\circ\psi_{2\pi}(p))=(2\pi, \phi^{-1}\circ\psi_{2\pi}(p)),$$
and by \ref{psi1}, 
$$\Phi(0,\phi(p))=(0,\psi_{2\pi}\circ\psi^{-1}_0\circ \phi^{-1}\circ \psi_0(\phi(p)))=(0,\psi_{2\pi}(p))=(0,\phi(\phi^{-1}\circ\psi_{2\pi}(p))).$$
Hence $\Phi$ descends to a smooth function on $MT(F,\phi)$. Since $\phi={\rm id}$ on $K$, by \ref{psi2}, $\Phi={\rm id}$ on $\partial MT(F,\phi)$. Hence $\Phi$ descends to a diffeomorphism on $OB(F,\phi)$. By definition $\Phi$ commutes with $\hat{\Theta}$. It is clear that $\phi^*\omega_x=\omega_x$ for all $x\in [0,2\pi]$. Hence by \ref{psi3}, we get
\begin{eqnarray*}
\Phi^*(\alpha_s/|\hat{h}|)_{|TF^0_x}
&=&\Phi^* \omega^s\\
&=&\big(\psi_{2\pi}\circ\psi^{-1}_x\circ \phi^{-1}\circ \psi_x\big)^*\omega^s\\
&=&\psi_x^*(\phi^{-1})^*(\psi^{-1}_x)^*\psi_{2\pi}^*\omega_{2\pi}\\
&=&\psi_x^*(\phi^{-1})^*(\psi^{-1}_x)^*\omega_0\\
&=&\psi_x^*(\phi^{-1})^*\omega_x\\
&=&\psi_x^*\omega_x\\
&=&\omega_0=\omega=d(\alpha/|h|)_{|TF^0_x}.
\end{eqnarray*} 
\end{proof}

Now our task is to establish that for a given $\epsilon>0$ the statements \ref{F1}-\ref{F5} hold for some $\alpha_s$. To this end we first study the Reeb vector field $R_s$ of $\alpha_s$ on $OB(F,\phi)\setminus W'$. We know that 
$d\eta$ is well-defined on $OB(\phi)\setminus W'$ and restricts to a symplectic form on each fibre $F^0_x$. We define 
the vector field $Y$ on $OB(\phi)\setminus W'$ so that it is tangent to $F_x$ for each $x$ and satisfies 
$$\imath_Yd\eta= -\beta'\eta_\Delta$$
along $F_x$ for each $x$.
Since ${\rm supp} (\phi)\subset ([0,\rho]\times K)^c$, $Y$ is compactly supported in $OB(\phi)\setminus W'$. We note that 
\begin{eqnarray*}
 2\pi(1+\delta)\imath_{(\partial_x+Y)} d\alpha_s
 &=&s \big(\imath_{(\partial_x+Y)}\beta'dx\wedge \eta_\Delta+\imath_{(\partial_x+Y)}d\eta\big)\\
 &=& s \big(\beta'dx(\partial_x+Y)\eta_\Delta-\eta_\Delta(\partial_x+Y)\beta'dx 
 +\imath_Y d\eta\big)\\
 &=&s(\beta'\eta_\Delta+d\eta (Y,Y)\beta' dx- \beta'\eta_\Delta)\\
 &=&0.
\end{eqnarray*}
Hence on $OB(\phi)\setminus W'$, the Reeb vector field of $\alpha_s$ reads as
\begin{eqnarray}\label{ReebsonW'comp}
 R_s=\frac{\partial_x+Y}{\alpha_s(\partial_x+Y)}.
\end{eqnarray}
We note that since $Y$ is supported in $OB(\phi)\setminus W'$ and $OB(\phi)\setminus W'$ is invariant under the flow of 
$\partial_x$, $OB(\phi)\setminus W'$ is invariant under the flow of $R_s$ and hence $W'$ is invariant as well. Moreover the $\partial_x$ component of $R_s$  never vanishes and since 
$Y$ is tangent $F_x$, $R_s$ is transverse to $F_x\setminus W'$ for all $x$. 

On $W'$, by (\ref{alpha_sgeneral}) we have
\begin{eqnarray}\label{ReebsonW'}
 R_s=2\pi(1+\delta)\frac{f'R_0-g'\partial_x}{f'g-fg'},
\end{eqnarray}
where $R_0$ is the Reeb vector field of $\alpha_0$ on $K$. We note that since $g'<0$ on $(0,\rho]$, 
$\partial_x$ component of $R_s$ does not vanish hence $R_s$ is transverse to $F_x$ for all $x$. For 
$r$ close to 0, by \ref{A3} we get
\begin{eqnarray}\label{ReebsnearK}
 R_s=2\pi(R_0+\partial_x)
\end{eqnarray}
and as $\partial_x$ vanishes on $K$, 
\begin{eqnarray}\label{ReebsonK}
 R_s=2\pi R_0
\end{eqnarray}
on $K$. Since $R_s$ leaves $K$ invariant  and  have a non-vanishing $\partial_x$ component, $F=F_0$ is a global 
hypersurface of sections for the flow of $R_s$ and we have well-defined first return time function and first return map, namely
$$\tau:F\setminus K\rightarrow (0,\infty),\; \varphi:F\setminus K \rightarrow F\setminus K.$$
We compute the flow of $R_s$ on $W'\setminus K$ via solving the system
$$\dot r=0,\; \dot x=\frac{-g'2\pi(1+\delta)}{f'g-fg'},\; \dot q=\frac{f'2\pi(1+\delta)}{f'g-fg'}R_0(q)$$
in the coordinates given by (\ref{coordinatesinW}). The components of the flow reads as
$$r(t)=r,\; x(t)=x-\frac{g'2\pi(1+\delta)}{f'g-fg'}t,\; q(t)=\phi_{R_0}^{\frac{f'2\pi(1+\delta)}{f'g-fg'}t}(q)$$
where $\phi^s_{R_0}$ is the flow of $R_0$ on $K$. Hence we have
\begin{eqnarray}\label{tauandvarphionW'}
\tau(r,q)=\frac{-f'g+fg'}{g'(1+\delta)},\;\varphi(r,q)=\big(r,\phi_{R_0}^{-2\pi\frac{f'}{g'}}(q)\big).
\end{eqnarray}
Hence for $r\leq r_0$, by \ref{A3}, we have
\begin{eqnarray}\label{nearK}
\tau=1,\; \varphi(r,q)=\big(r,\phi_{R_0}^{2\pi}(q)\big). 
\end{eqnarray}
Hence $\tau$ and $\varphi$ extend smoothly to $K$, that is, the statement \ref{F2} is satisfied. We note that 
$$2\pi R_0=R'$$
where $R'$ is the Reeb vector field of $\alpha'$ on $K$. Now we invoke our assumptions on the contact form $\alpha'$ on $K$. By \ref{F'3}, we have a circle bundle over the open subset $U'\subset F'$, which is given by the closed Reeb orbits of $\alpha'$. Let $V'$ be the total space of this circle bundle. Then $V'$ is an open subset of $K$ and  $\phi_{R'}^{1}(q)=q$ for all $q\in V'$. Hence we have $\phi_{R_0}^{2\pi}(q)=q$ for all $q\in V'$. We define the open subset 
\begin{equation}\label{defnofU}
U:=(0,r_0)\times V'\subset F.
\end{equation} 
Then (\ref{nearK}) leads to
\begin{equation}\label{tauphionU}
\tau_{|U}=1,\; \varphi_{|U}={\rm id},
\end{equation}
which establishes \ref{F3}. 

Now we want to show that suitable choices of $\epsilon'$, $\delta$ and $s>0$ lead to \ref{F4} and \ref{F5}. 

We first observe the following.  
\begin{lem}\label{taubound} There exists $s_3\in (0,s_2)$ such that for all $s\in (0,s_3)$, 
$$||\tau-1||_\infty<2\delta.$$ 
\end{lem}
\begin{proof} By (\ref{tauandvarphionW'}), $\tau$ depends only on $r$ on $(0,\rho]\times K$ and by \ref{A5}, we have
$$\frac{\partial \tau}{\partial r}=\frac{g(f'g''-f''g')}{(g')^2(1+\delta)}\geq 0.$$
For any $s$, by \ref{A1} we have 
$$\tau=\frac{1}{1+\delta}$$
near $\rho$. Hence for all $s$
$$\frac{1}{1+\delta}\leq \tau\leq 1$$
on $[0,\rho]\times K$. Hence 
$$\sup_{[0,\rho]\times K}|\tau-1|\leq \frac{1}{1+\delta}-1\leq \delta$$
for any $s\in (0,s_2)$.

We note that on $OB(F,\phi)\setminus W'$, we get
$$dx(R_s)=\frac{1}{\alpha_s(\partial_x+Y)}$$
using (\ref{ReebsonW'comp}) and the fact that $Y$ is tangent to the fibres of $OB(F,\phi)$. We note that as $s\rightarrow 0$
$$\alpha_s(\partial_x+Y)=\frac{1}{2\pi(1+\delta)}\tilde{\alpha}_s(\partial_x+Y)=\frac{1}{2\pi(1+\delta)}\bigg(1+s\big((1-\beta)\eta(Y)+\beta(\phi^*\eta)(Y)\big)\bigg)$$
converges to $\frac{1}{2\pi(1+\delta)}$ uniformly on $OB(F,\phi)\setminus W'$  since $Y$ is compactly supported. Hence $\tau$ converges to $\frac{1}{(1+\delta)}$ uniformly as $s\rightarrow 0$. 
Combining this with the above estimates proves the lemma.
\end{proof}

Now we want to estimate the $d\alpha_s-$volume of the complement of $U$ in $F$. 
\begin{lem}\label{alphasareaafterrho}
There exists $s_4\in (0,s_3)$ such that for all $s\in(0,s_4)$,
$$\int_{([0,\rho)\times K)^c}(d\alpha_s)^n<2\delta.$$
\end{lem}
\begin{proof}
We recall from the proof of Lemma \ref{lemcontact} that on $([0,\rho)\times K)^c$,
$$2\pi (1+\delta)(d\alpha_s)_{|TF}=s \big(\beta'dx\wedge \eta_\Delta+d\eta\big)_{|TF}=sd\eta.$$
Hence  the integral 
$$\int_{([0,\rho)\times K)^c}(d\alpha_s)^n\searrow 0 \textrm{ as } s\rightarrow 0.$$
\end{proof}
Next we consider the region $[r_0,\rho]\times K$.
We have
$$2\pi (1+\delta)(d\alpha_s)_{|TF}=\big(f'\,dr\wedge dx+(g'dr\wedge\alpha_0+g\,d\alpha_0)\big)_{|TF}=g'dr\wedge\alpha_0+g\,d\alpha_0$$
$$\Rightarrow\; (2\pi (1+\delta))(d\alpha_s)^n_{|TF} =(n-1)g'g^{n-1}\,dr\wedge \alpha_0\wedge (d\alpha_0)^{n-1}.$$
We note that $dr\wedge \alpha_0\wedge (d\alpha_0)^{n-1}$ is a negative volume form due to the parametrization (\ref{lambdanearK}). Hence 
\begin{eqnarray*}
\int_{([r_0,\rho)\times K)}(d\alpha_s)^n
&=&\frac{n-1}{(2\pi)^n(1+\delta)^n}\int_{([r_0,\rho)\times K)}g'g^{n-1}\,dr\wedge \alpha_0\wedge (d\alpha_0)^{n-1}\\
&=&\frac{n-1}{(2\pi)^n(1+\delta)^n}{\rm vol}(K,\alpha_0)\int_{r_0}^\rho (-g'g^{n-1})\,dr\\
&\leq & \frac{n-1}{(2\pi)^n(1+\delta)^n} {\rm vol}(K,\alpha_0)(1+\delta)^{n-1} \int_{r_0}^\rho (-g')\,dr\\
&=&\frac{n-1}{(2\pi)^n (1+\delta)} {\rm vol}(K,\alpha_0) (g(r_0)-g(\rho))\\
&\leq & \frac{n-1}{(2\pi)^n} {\rm  vol}(K,\alpha_0)2\delta\\
&=&(n-1){\rm vol}(K,\alpha')2\delta.
\end{eqnarray*}
The remaining piece of $F$ is the subset $(0,r_0)\times (K\setminus V')$. We note that 
\begin{eqnarray*}
\int_{(0,r_0)\times (K\setminus V')} (d\alpha_s)^n
&=&\frac{1}{(2\pi)^n(1+\delta)^n}\int_{(0,r_0)\times (K\setminus V')}(2r\,dr\wedge dx-2r\,dr\wedge \alpha_0+(1+\delta-r^2)\,d\alpha_0)^n\\
&=& \frac{(n-1)}{(2\pi)^n(1+\delta)^n}\int_{(0,r_0)\times (K\setminus V')}-2r(1+\delta-r^2)^{n-1}\,dr\wedge\alpha_0\wedge (d\alpha)^{n-1}\\
&=& \frac{(n-1)}{(2\pi)^n(1+\delta)^n}{\rm vol}(K\setminus V',\alpha_0)\int_0^{r_0}2r(1+\delta-r^2)^{n-1}\,dr\\
&=& \frac{(n-1)}{n(2\pi)^n(1+\delta)^n }{\rm vol}(K\setminus V',\alpha_0) (1+\delta-r^2)^n\big|^0_{r_0}\\
&\leq &\frac{(n-1)}{n(2\pi)^n(1+\delta)^n }{\rm vol}(K\setminus V',\alpha_0) (1+\delta)^n\\
&\leq &\frac{{\rm vol}(K\setminus V',\alpha_0)}{(2\pi)^n}\\
&=&{\rm vol}(K\setminus V',\alpha'). 
\end{eqnarray*}
In order to estimate the above volume, we use \ref{F'4},  \ref{F'5} and get
\begin{eqnarray*}
{\rm  vol}(K\setminus V',\alpha')=\int _{F'\setminus U'}\tau'\, (d\alpha')^{n-1} \leq (1+\epsilon')\epsilon'
\end{eqnarray*}
and hence 
\begin{equation}
 \int_{(0,r_0)\times (K\setminus V')} (d\alpha_s)^n\leq 2\epsilon'
\end{equation}
if $\epsilon'<1/2$.

Now given $\epsilon>0$, we first choose $\epsilon'>0$ such that 
\begin{equation}\label{chooseepsilon'}
\epsilon'<\min\{\epsilon/6, 1/2\}.
\end{equation}
Then we get the contact form $\alpha'$ on $K$ satisfying \ref{F'1}-\ref{F'5} with respect to $\epsilon'$. Then using the definition (\ref{alpha_0}) and the parametrization (\ref{lambdanearK}) , we construct $\alpha_s$, where we pick $\delta>0$ satisfying 
\begin{equation}\label{choosedelta}
 \delta<\min\bigg\{\frac{\epsilon}{6(n-1) {\rm vol}(K,\alpha')},\frac{\epsilon}{6}\bigg\}.
\end{equation}
We note that once $\epsilon'>0$ is fixed, the 1-form $\alpha'$ on $K$ is fixed.  

By Lemma \ref{alphasareaafterrho}, there is some $s_4\in(0,s_3)$ such that for all $s\in (0,s_4)$,
\begin{eqnarray*}
\int_{F\setminus U} (d\alpha_s)^n
&\leq & \int_{([0,\rho)\times K)^c}(d\alpha_s)^n+\int_{([r_0,\rho)\times K)}(d\alpha_s)^n+\int_{(0,r_0)\times (K\setminus V')} (d\alpha_s)^n\\
&\leq & 2\delta+ (n-1) {\rm vol}(K,\alpha')2\delta+2\epsilon'\\
&<& \frac{\epsilon}{3}+\frac{\epsilon}{3}+\frac{\epsilon}{3}=\epsilon,
\end{eqnarray*} 
where the second line follows from the above estimates and the last line follows from (\ref{chooseepsilon'}) and (\ref{choosedelta}). 
This establishes \ref{F5}. 
We note that the estimate (\ref{choosedelta}) implies 
$$2\delta<\min\bigg\{\epsilon, \frac{\epsilon}{\,{\rm vol}(K,\alpha_s)}\bigg\}.$$
Hence from Lemma \ref{taubound}, it follows that for any  $s\in (0,s_4)$,
$$||\tau-1||_\infty<\min\bigg\{\epsilon, \frac{\epsilon}{\,{\rm vol}(K,\alpha_s)}\bigg\}.$$
Hence the statement \ref{F4} holds.
\end{proof}

Now we are ready to prove our main theorem. 
\begin{thm}\label{bigmanifolds}
Let $(V,\xi)$ be a closed connected contact manifold and let $C>0$ be given. Then there exists a contact form $\alpha$ on $V$ such that $\ker \alpha=\xi$ and 
$$\rho(V,\alpha)\geq C.$$
\end{thm}
\begin{proof}
Let $F\subset V$ be the hypersurface  with $\partial F=K$ given in Proposition \ref{niceform} and let $\epsilon>0$ be given. Then there is a contact form $\alpha$ on $(V,\xi)$ such that the statements \ref{F1}-\ref{F5} hold. Let $\dim V=2m+1$.

We put
$$c:=\int_F(d\alpha)^m=\int_K\alpha \wedge (d\alpha)^{m-1},\;a:=\int_U(d\alpha)^m$$
so that by \ref{F5}, 
$$c-a<\epsilon.$$

Let $\epsilon'>0$ be given. We claim that there is a finite collection of smooth embeddings
$$\psi_j:r_j\mathbb{B}\rightarrow U,\; \psi_j^*\omega=d\alpha; \; j=1,...,k,$$
where $\mathbb{B}$ is the open unit ball of dimension $2m$ and $\omega$ is the standard symplectic form, such that the images of these embeddings are mutually disjoint and 
$$\sum_j {\rm vol}(r_j\mathbb{B},\omega)=(1-\epsilon')a.$$
We note that the open subset $U$ is precompact in $F$. So we can take a finite open cover of $U$ consisting of Darboux coordinate charts. After making this finite collection mutually disjoint, it is enough to fill each corresponding open subset of $\mathbb{R}^{2m}$ with balls of arbitrary center and radius so that the total volume is $(1-\epsilon')a$. 

Using \ref{F2} and \ref{F3}, we define the embeddings 
$$\Psi_j:\mathbb{R}/\mathbb{Z}\times r_j\mathbb{B}\rightarrow V$$
such that $\Psi_j(0,\cdot)=\psi_j$ and $\Psi_j^*R_\alpha=\partial_s$ where $s\in \mathbb{R}/\mathbb{Z}$. We claim that 
$$\alpha_j:=\Psi_j^*\alpha=\lambda_j+ds$$
where $\lambda_j$ is a primitive of $\omega$ on $r_j\mathbb{B}$. By construction, $\partial_s$ is the 
Reeb vector field of $\alpha_j$. Hence $\alpha_j$ is invariant under the translation along the 
$s$ coordinate and has the form
$\alpha_j=\lambda_j+uds$
where $\lambda_j$ is a 1-form on $r_j\mathbb{B}$ and $u$ is a function on $r_j\mathbb{B}$. Since $\alpha_j(\partial_s)=1$
one has $u\equiv 1$ and again by construction
$$\omega=\psi_j^*d\alpha=\Psi_j^*d\alpha|_{r_j\mathbb{B}\times \{0\}}=d\alpha_j|_{r_j\mathbb{B}\times \{0\}}=
d\lambda_j.$$
We denote the image of $\Psi_j$ by $W_j$. We have
\begin{eqnarray*}
 \sum _j {\rm vol}(W_j,\alpha)
 &=& \sum _j {\rm vol}(\mathbb{R}/\mathbb{Z}\times r_j\mathbb{B}, \alpha_j)\\
 &=& \sum _j {\rm vol}(\mathbb{R}/\mathbb{Z}\times r_j\mathbb{B}, \lambda_j+ds)\\
 &=& \sum _j {\rm vol}(r_j\mathbb{B},\omega)\\
 &=&(1-\epsilon')a.
\end{eqnarray*}
Using \ref{F4}, we get
$$\tau\leq 1+\frac{\epsilon}{c}$$
$$\Rightarrow\; {\rm vol}(V,\alpha)=\int_F\tau\,(d\alpha)^m\leq (1+\frac{\epsilon}{c})c=c+\epsilon.$$ 
Using the fact that $W_j$'s are disjoint, we get 
\begin{eqnarray*}
{\rm vol}(V\setminus \bigcup_jW_j,\alpha)
&=&{\rm vol}(V,\alpha)-\sum_j {\rm vol}(W_j,\alpha)\\
&\leq & (c+\epsilon)-(1-\epsilon')a \\
&=&(c-a)+\epsilon +\epsilon' a\\
&\leq &\epsilon +\epsilon  +\epsilon' a\\
&=&2\epsilon +\epsilon' a.
\end{eqnarray*}

Now for each $j$, we apply Lemma \ref{plug_any_primitive}, with $r=r_j$, $\lambda'=\lambda_j$ and 
get the contact form $\beta_j$ on $r_j\mathbb{B}\times \mathbb{R}/\mathbb{Z}$ satisfying \ref{p'1}-\ref{p'3} and 
$${\rm vol}(r_j\mathbb{B}\times \mathbb{R}/\mathbb{Z},\beta_j)<\frac{\epsilon'}{k}.$$
We define $\alpha_j=\Psi_j^*\beta_j$ on $W_j$. By construction, $\alpha_j$'s fit $\alpha$ near the boundary of 
$W_j$ and define a contact form $\hat{\alpha}$ on $V$. Since $\hat{\alpha}$ is isotopic to $\alpha$, 
$\ker \hat{\alpha}$ is diffeomorphic to $\ker \tilde{\alpha}$. We have the estimate
\begin{eqnarray*}
{\rm vol}(V,\hat{\alpha})
&=&{\rm vol}(V\setminus \bigcup_jW_j,\hat{\alpha})+\sum_j {\rm vol}(W_j,\hat{\alpha})\\
&=& {\rm vol}(V\setminus \bigcup_jW_j,\alpha)+\sum_j {\rm vol}(W_j,\beta_j)\\
&\leq &  2\epsilon +\epsilon' a +\epsilon'\\
&=&2\epsilon+\epsilon'(a+1).
\end{eqnarray*}
We note that the open subsets $W_j$ are invariant under the Reeb flow of $\hat{\alpha}$. The closed orbits that are contained in $W_j$'s have period at least 1 due to the construction of contact forms $\{\beta_j\}$s. Due to \ref{F1}, closed orbits of $\hat{\alpha}$, which are contained in $K$ have period at least $1/2$. Taking $\epsilon<1/2$, the remaining closed orbits have period at least $1/2$ due to \ref{F4}. Hence we get 
$$T_{{\rm min}}(\hat{\alpha})\geq 1/2$$
and therefore 
$$\rho(V,\hat{\alpha})=\frac{(T_{{\rm min}}(\hat{\alpha}))^{m+1}}{{\rm vol}(V,\hat{\alpha})}\geq \frac{(1/2)^{m+1}}{2\epsilon+\epsilon'(a+1)}.$$
Now given $C>0$, taking $\epsilon$ and $\epsilon'$ small enough yields 
$$\rho(V,\hat{\alpha})\geq C.$$
We note that once $\epsilon$ is chosen, $a$ is fixed and one may choose $\epsilon'$ for the given $a$. 
\end{proof}


\begin{thebibliography}{9}
\bibitem[ABHS18a]{3sphere}
A. Abbondandolo, B. Bramham, U. L. Hryniewicz and P. A. S. Salom\~ao, \emph{Sharp systolic inequalities for Reeb flows on the three-sphere}, Invent. Math. 211 (2018), 687-778.
 
\bibitem[ABHS18b]{general}
 A. ~Abbondandolo, B. Bramham, U. L. Hryniewicz and P. A. S. Salom\~ao, 
 \emph{Contact forms with large systolic ratio},  Annali della Scuola Normale di Pisa - Classe di Scienze, (to appear).

\bibitem[APBT16]{APBT}
J. C. Alvarez Paiva, F. Balacheff, and K. Tzanev,
\emph{Isosystolic inequalities for optical
hypersurfaces}, 
Adv. Math. 301, 2016.

\bibitem[BK18]{Benedetti}
G. Benedetti and  J. Kang, 
\emph{A local systolic-diastolic inequality in contact and symplectic geometry}, 
Preprint, arXiv:1801.00539.

\bibitem[Cro88]{croke}
 C. B. Croke, 
 \emph{Area and length of the shortest closed geodesic}, J. Differential Geom. 18, 1988.

\bibitem[Gir03]{openbook} E. Giroux, \emph{Géométrie de contact: de la dimension trois vers les dimensions supérieures},
Preprint, arXiv:math/0305129.

\bibitem[Gir17]{ILD} E. Giroux, \emph{Ideal liouville domains - a cool gadget}, 
Preprint, arXiv:1708.08855.


\bibitem[Gro83]{Gromov} 
M. Gromov, 
\emph{Filling Riemannian manifolds},
J. Differential Geom. 18, 1983.


\bibitem[McS95]{Mcduff}
D. McDuff and D. Salamon, 
\emph{Introduction to symplectic topology},
Oxford mathematical monographs, 1995.


\bibitem[Tau07]{Taubes}
C. H. Taubes, 
\emph{The Seiberg-Witten equations and the Weinstein conjecture}, Geom. Topol. 11, 2007.

 \end{thebibliography}
\end{document}